\documentclass[12pt,a4paper]{article}

\usepackage{amsmath,amsthm,amssymb,a4wide}
\usepackage{times}
\usepackage{enumerate}

\usepackage{setspace}

\usepackage{dsfont}
\DeclareMathOperator{\Int}{Int}

\DeclareMathOperator{\tr}{tr}

\newcommand{\R}{\mathbb R}

\newcommand{\gep}{\varepsilon}
\newcommand{\RR}{\mathbb{R}}
\newcommand{\RRn}{\mathbb{R}^n}

\newcommand{\open}{(\Omega_t)_{t\in(0,T)}}

\newcommand{\omegaone}{(\Omega^1_t)_{t\in(0,T)}}
\newcommand{\omegatwo}{(\Omega^2_t)_{t\in(0,T)}}
\newcommand{\close}{(\mathcal{F}_t)_{t\in(0,T)}}
\newcommand{\reach}{\mathcal{R}_{\hat x,\hat t}}
\newcommand{\reachtwo}{\mathcal{R}_{\hat x,\hat t-h}}
\newcommand{\liminfstar}{{\liminf_*}_{\varepsilon\to0^+}\;}
\newcommand{\limsupstar}{{\limsup^*}_{\varepsilon\to0^+}\;}
\newcommand{\limsupstard}{{\limsup_{\varepsilon\to0^+}}^*\;}
\newcommand{\liminfstard}{\liminf_{\varepsilon\to0^+}\hspace{-0.08cm}{}_*\;}
\newcommand{\qued}{q^{\varepsilon,\delta}}
\newcommand{\omegaeb}{\omega^{\varepsilon,\beta}}
\providecommand{\norm}[1]{\lVert#1\rVert}
\providecommand{\abs}[1]{\lvert#1\rvert}
\providecommand{\bigabs}[1]{\Big\lvert#1\Big\rvert}

\theoremstyle{definition}
\newtheorem{thm}{Theorem}[section]
\newtheorem{cor}[thm]{Corollary}
\newtheorem{prop}[thm]{Proposition}
\newtheorem{lem}[thm]{Lemma}
\newtheorem{defin}[thm]{Definition}
\newtheorem{rem}[thm]{Remark}

\newcommand{\subjclass}[1]{\bigskip\noindent\emph{2010 Mathematics Subject Classification:}\enspace#1}

\numberwithin{equation}{section}

\begin{document}

\title{Singular limits of reaction diffusion equations and geometric flows with discontinuous velocity}
\author{Cecilia De Zan \& Pierpaolo Soravia\thanks{email: soravia@math.unipd.it.  }\\
Dipartimento di Matematica\\ Universit\`{a} di Padova, via Trieste 63, 35121 Padova, Italy}

\date{}
\maketitle

\begin{abstract}
We consider the singular limit of a bistable reaction diffusion equation in the case when the velocity of the traveling wave solution depends on the space variable and converges to a discontinuous function.
We show that the family of solutions converges to the stable equilibria off a front propagating with a discontinuous velocity. The convergence is global in time by applying the weak geometric flow uniquely defined through the theory of viscosity solutions and the level-set equation.
\end{abstract}


\section{Introduction}

Many phenomena in physics, chemistry, biology etc. give rise to moving interfaces. In mathematics these are sometimes modeled by reaction diffusion equations whose solution, often an {\it order parameter}, is expected to approach for large times the equilibria of the system. When there is more than one equilibrium, interfaces separate regions where the parameter tends to the different equilibria, called phases for instance in phase transition models.
In this paper we want to study globally in time, as $\gep\downarrow0$, the asymptotic behavior of the following  reaction diffusion equation
\begin{equation}\label{asymptotic-probl}
\left\{
\begin{array}{lll}
(\mbox i)&u^{\varepsilon}_t(x,t)-\varepsilon\Delta u^\varepsilon(x,t)+\varepsilon^{-1}f^\varepsilon(u^\varepsilon,x)=0&\mbox{  in }\RRn\times(0,+\infty),\\
(\mbox{ii})&u^{\varepsilon}(x,0)=g(x)&\mbox{  in }\RRn,
\end{array}
\right.
\end{equation}
when $f^\varepsilon:\RR\times\RRn\longrightarrow\RR$ is of bistable type, with structure conditions modelled on the following main example
\begin{equation}\label{fepsilon def}
f^\varepsilon(q,x):=2\Big(q-\frac{c^\varepsilon(x)}{2}\Big)(q^2-1)
\end{equation}
with $-1<c^\gep(x)/2<1$. It is known in the literature and proved by Barles-Soner-Souganidis \cite{bss}, that if the bounded family of smooth functions $\{c^\gep\}_{\varepsilon>0}\subset C^{1,1}(\R^n)$, which are the velocities of the traveling wave solutions of (\ref{asymptotic-probl}), converges locally uniformly to some continuous function $\alpha:\R^n\to\R$ and the initial condition $g$ represents a sharp interface across the unstable equilibrium, then the asymptotics is governed by the following geometric Hamilton-Jacobi equation
\begin{equation}\label{eqhj}
\left\{\begin{array}{ll}
u_t(x,t)+\alpha(x)|Du(x,t)|=0,\quad \R^n\times(0,+\infty)\\
u(x,0)=u_o(x).
\end{array}\right.
\end{equation}
Here the initial condition $u_o\in C(\R^n)$ is chosen in such a way that the initial front $\Gamma_o=\{x\in\R^n:u_o(x)=0\}=\{x\in\R^n:g(x)=\frac{\alpha(x)}2\}$ and $\Gamma_o$ is a nonempty and closed set (ideally an hypersurface). Moreover $u_o(x)>0$ (resp. $u_o(x)<0$) if $g(x)>\frac{\alpha(x)}2$ (resp. $g(x)<\frac{\alpha(x)}2$). Indeed one proves that the convergence occurs locally uniformly off the moving front determined by (\ref{eqhj}) to the stable equilibria of the reaction diffusion equation, namely
\begin{equation*}u^\varepsilon(x,t)\to
\left\{\begin{array}{ll}
1,\quad &\hbox{if } u(x,t)>0,\\
-1,&\hbox{if } u(x,t)<0,
\end{array}
\right.\end{equation*}
where $u$ is the solution of (\ref{eqhj}).
We recall that, in order to solve (\ref{eqhj}) globally in time, solutions are meant as viscosity solutions, see Crandall-Ishii-Lions \cite{cil}.
It turns out that (\ref{eqhj}) has a unique continuous solution $u\in C(\R^n\times[0,+\infty))$ for any $u_o\in C(\R^n)$.
Such equation is called geometric since by homogeneity of the operator with respect to the first derivatives of $u$, one proves that if $u$ solves the pde in (\ref{eqhj})
and $\psi:\R\to\R$ is smooth and increasing, then also $\psi(u)$ solves the same equation. As a consequence, it is easy to see that if $u^1_o$ and $u^2_o$ are two initial conditions such that
$$\{x:u^1_o(x)=0\}=\{x:u^2_o(x)=0\},$$
and $u^1, u^2$ are the corresponding solutions in (\ref{eqhj}), then one has
\begin{equation*}
{\{x:u^1(x,t)=0\}=\Gamma_t=\{x:u^2(x,t)=0\}},\quad \mbox{for all }t>0.
\end{equation*}
One can therefore {\it define} the family of closed sets $(\Gamma_t)_t$ to be the geometric flow of the front or interface $\Gamma_o$ with {normal velocity $-\alpha$}.

In a previous paper \cite{dzs1}, we proved that the problem (\ref{eqhj}) is well posed, and a comparison principle holds in the sense of viscosity solutions as defined by Ishii \cite{is} (that we recall below) also when $\alpha$ has constant sign and it is piecewise continuous across an hypersurface, see also Camilli \cite{ca}. In the present paper we will apply these results to (\ref{asymptotic-probl}) allowing the sequence $c^\varepsilon$ to only converge off an hypersurface.
The novelty of our study is that in our case the norms of the gradients $\|Dc^\varepsilon\|_\infty$, $\|D^2c^\varepsilon\|_\infty$ may blow up as $\varepsilon\to0$, see (\ref{fproperty2}), (\ref{eqasscgep}) below.
Nonetheless we can still determine the asymptotic behavior of (\ref{asymptotic-probl}) for a general initial condition. We will show that the family $u^\varepsilon$ converges
to the stable equilibria of (\ref{asymptotic-probl})
off the evolving interface which moves with normal velocity $-\alpha$, now discontinuous in space, and it is determined by the geometric equation (\ref{eqhj}), once we initialize it by setting, in the case (\ref{fepsilon def}),
\begin{equation*}
\Gamma_o=\{x\in\R^n:u_o(x)=0\}=\{x\in\R^n:\frac{\alpha_*(x)}2\leq g(x)\leq \frac{\alpha^*(x)}2\},
\end{equation*}
where $\alpha_*,\alpha^*$ indicate the lower and upper semicontinuous envelopes of $\alpha$, respectively.
We notice that $\Gamma_o$ may contain relatively open subsets of the hypersurface of discontinuity of $\alpha$ where $\frac{\alpha_*(x)}2< g(x)<\frac{\alpha^*(x)}2$.
In geometric optics, discontinuous coefficients $\alpha$ in the propagation equation (\ref{eqhj}) arise in the refraction phenomenon and $1/\alpha$ is then the discontinuous refraction index. This makes our study interesting for the applications.

In order to prove the convergence of the solutions of (\ref{asymptotic-probl}), we apply the general geometric approach in Barles-Souganidis \cite{bs} to study singular limits giving rise to moving interfaces.
Their approach has already been used to describe geometric flows also in KPP-type systems, equations with oscillating coefficients, nonlocal terms or appearing in the study of interacting particle systems, see also Souganidis \cite{soug}. We show that it can be adapted also in our case.
The approach in \cite{bs} is based on an equivalent definition of weak geometric flow through the local comparison with smooth evolutions, as we recall below. This fact allows to apply more directly the formal arguments, where the smoothness of the interface and of the solution of the geometric equation is assumed, in order to derive the asymptotics.
In our discussion, we are going to follow the approach of \cite{bs}, as revisited by Barles-Da Lio \cite{bdl}, where they study problems in bounded domains with a Neumann boundary condition. We will often adapt to our problem a combination of the arguments of these two papers. To implement a general initial condition, we also need to follow some ideas of Chen \cite{xce} in order to show that an interface initializes in short time. We recall here also the work by Da Lio, Kim, Slepčev \cite{dlks}, where they study the asymptotics of a reaction diffusion equation with a nonlocal term, with a scaling different than ours, in a bounded domain with a nonlinear oblique derivative boundary condition. As we mentioned, the general approach in \cite{bs,bdl} does not apply directly in our case, and to cope with the discontinuous velocity of the front we also need to use an equivalent definition of solution of (\ref{eqhj}) by using one sided continuous approximations of the velocity, an idea already used in \cite{dlks}. We will also show that, when $\alpha$ in (\ref{eqhj}) has a sign and the initial front has empty interior, then the no interior condition persists for all times, thus avoiding a possible unpleasant feature of the weak evolution.

We can also consider a different scaling in the reaction diffusion equation, namely
\begin{equation}\label{asymptotic-probl-two}
\left\{
\begin{array}{lll}
(\mbox {iii})&u^{\varepsilon}_t(x,t)-\Delta u^\varepsilon(x,t)+\varepsilon^{-2}f^\varepsilon(u^\varepsilon,x)=0&\mbox{  in }\RRn\times(0,+\infty),\\
(\mbox{iv})&u^{\varepsilon}(x,0)=g(x)&\mbox{  in }\RRn.
\end{array}
\right.
\end{equation}
rather than (\ref{asymptotic-probl}). In this case, if ${c^\gep}/\gep\to\alpha$, with $\alpha$ piecewise continuous across an hypersurface, and we can prove that equation (\ref{asymptotic-probl-two}) as $\gep\to0$ gives rise to an interface moving with normal velocity ${\cal K}-\alpha$, where $\cal K$ indicates the mean curvature of the interface. Thus the front moves according to the geometric equation
\begin{equation}\label{eqmc}
\left\{\begin{array}{ll}
u_t(x,t)+F(Du(x,t),D^2u(x,t))+\alpha(x)|Du(x,t)|=0,\quad(x,t)\in\RRn\times(0,+\infty),\\
u(x,0)=u_o(x),
\end{array}\right.
\end{equation}
where $F:\RRn\times\mathcal{S}^n\to\RR$ is defined as
\begin{equation}\label{mean curvature F}
F(p,X)=-\tr\Big[\Big(I-\frac{p}{\abs{p}}\otimes\frac{p}{\abs{p}}\Big)X\Big].
\end{equation}
We can prove the convergence of the family $(u_\varepsilon)_{\varepsilon>0}$ also in this case, provided (\ref{eqmc}) satisfies a comparison principle. At the present time, as far as we know, a general comparison principle for (\ref{eqmc}) when $\alpha$ is piecewise continuous does not yet appear in the literature. We proved however a positive result in bounded domains in \cite{dzs3}.

We finally recall that the so called level set method for geometric flows was proposed by Osher-Sethian \cite{os} for numerical computations of geometric flows. Equations (\ref{eqhj}), (\ref{eqmc}) are main examples of their theory.
The rigorous theory of weak front evolution started with the work by Evans-Spruck \cite{es} for the mean curvature flow and by Chen-Giga-Goto \cite{cgg} for more general geometric flows.
For the mathematical analysis of the level set method via viscosity solutions, the reader is referred to the book by Giga \cite{gi}, where the approach is discussed in detail.
Among others, one of the most striking applications of the theory of weak front propagation is the fact that it allows to rigorously determine the asymptotics of reaction diffusion equations and sytems which model phase transitions. In this regard equation (\ref{asymptotic-probl-two}) (with $x-$independent nonlinearity $f$) was proposed by Allen-Cahn \cite{ac} as a phase transition model for a moving interface with normal velocity being the mean curvature of the front. The first study of the Allen-Cahn equation with a formal asymptotics is by Keller-Rubinstein-Sternberg \cite{krs} and the first rigorous and global in time proof of the asymptotics is due to Evans-Soner-Souganidis \cite{ess}. 
An application of the level set method to study the asymptotics of a reaction diffusion system appears in Soravia-Souganidis \cite{ss}.

As a general notation, in the paper we denote by $B(x,r),B(x,r]$ the open and closed balls in $\R^n$ with center $x$ and radius $r\geq0$, respectively.

\section{Definitions and basic properties}

In this section we consider a measurable function $\alpha:\RRn\to[\rho,+\infty)$, $\rho>0$, which is bounded and piecewise continuous across a given oriented, closed, Lipschitz hypersurface
$\tilde\Gamma\subset\RRn$ as follows. We are given
two bounded and locally Lipschitz continuous functions $n_1,n_2:\RRn\longrightarrow[\rho,+\infty)$ such that $n_1(x)< n_2(x)$, for all $x\in\RRn$.
If we denote with $\tilde d$ a signed distance function from $\tilde\Gamma$, then we consider $\alpha$ such that
\begin{equation}\label{alpha def}
\alpha(x)\in\left\{
\begin{array}{ll}
\{n_1(x)\}&\mbox{if }\tilde d(x)<0,\\
\{n_2(x)\}&\mbox{if }\tilde d(x)>0,\\
\left[n_1(x),n_2(x)\right]&\mbox{if }\tilde d(x)=0.
\end{array}
\right.
\end{equation}

We first briefly recall the basic ideas and results of the level-set approach, for the details see \cite{bss,soug,gi} and the references therein.

Let $\mathcal{E}$ be the collection of all the triples $(\Gamma_o, D^+_o, D^-_o)$ of mutually disjoint subsets of $\RRn$ such that $\Gamma_o$ is closed, $D^{\pm}_o$ are open and $\RRn=\Gamma_o\cup D^+_o\cup D^-_o$. We choose a function $u_o\in C(\RRn)$ such that
$$D_o^+=\{x\in\RRn:u_o(x)>0\},\quad D_o^-=\{x\in\RRn:u_o(x)<0\},\quad \Gamma_o=\{x\in\RRn:u_o(x)=0\}.$$
Given $\alpha$ as above, in order to define the weak motion or geometric flow of $(\Gamma_o, D^+_o, D^-_o)$ by normal velocity $-\alpha$ we start by considering the viscosity solution $u\in {C}(\RRn\times[0,+\infty))$ of the Cauchy problem (\ref{eqhj}).
All of what we are stating below in this section would also hold true for the other interesting geometric equation (\ref{eqmc}), in the case of a geometric flow with normal velocity ${\cal K}-\alpha$, provided it satisfies a comparison principle. This problem is not completely solved in the literature although we solve it in bounded domains in \cite{dzs3}.

We recall that, following Ishii \cite{is}, a locally bounded viscosity solution $u:\R^n\times(0,+\infty)\to\R$ of the pde in (\ref{eqhj}) is defined by checking the {two differential inequalities}
\begin{eqnarray*}
u_t(x,t)+{\alpha_*(x)}|Du(x,t)|\leq 0,\\
u_t(x,t)+{\alpha^*(x)}|Du(x,t)|\geq 0,
\end{eqnarray*}
in the viscosity sense, see \cite{cil}.
Here $\alpha^*(\hat x)=\lim_{r\to0^+}\sup_{B(\hat x,r)}\alpha(x)$ is the upper semicontinuous envelope, and the lower semicontinuous envelope $\alpha_*$ is defined accordingly. For instance, whenever $\varphi\in C^1(\R^n\times(0,+\infty))$ and $u^*-\varphi$ has a local maximum point at $(x_o,t_o)$, then
\begin{equation*}
\varphi_t(x_o,t_o)+{\alpha_*(x_o)}|D\varphi(x_o,t_o)|\leq 0.
\end{equation*}
A locally bounded function $u:\R^n\times[0,+\infty)$ will be a (discontinuous) solution of (\ref{eqhj}) if it is moreover continuous at the points of $\{(x,0):x\in\R^n\}$ and $u(x,0)=u_o(x)$.
It is known that, for every $u_o\in C(\RRn)$ there exists a unique solution $u\in C(\R^n\times[0,+\infty))$ of (\ref{eqhj}). For this fact the reader can consult the standard theory in \cite{cil} when $\alpha$ is continuous, or \cite{ca,dzs1} and the references therein for a discontinuous $\alpha$.
If for $t>0$, we define the triple
$$D_t^+:=\{x\in\RRn:u(x,t)>0\},\quad D_t^-:=\{x\in\RRn:u(x,t)<0\},\quad \Gamma_t:=\{x\in\RRn:u(x,t)=0\},$$
we have that $(\Gamma_t,D_t^+,D_t^-)\in\mathcal{E}$ for all $t\geq0$, and, since the equation in (\ref{eqhj}-i) is geometric as recalled in the introduction, the collection $\{(\Gamma_t,D_t^+,D_t^-)\}_{t\geq0}$ is uniquely determined, independently of the choice of the initial datum $u_o$ with the properties above, by the initial triple $(\Gamma_o,D_o^+,D_o^-)$.

One of the interesting facts of weak geometric flows in the level set approach is that even if we start out with a smooth initial hypersurface $\Gamma_o$, at some later time $t>0$, the front $\Gamma_t$ may develop interior points.
We say below that the \emph{no-interior condition} holds for the set $\{u=0\}$ if
\begin{equation}\label{nointerior}
\{(x,t):u(x,t)=0\}=\partial\{(x,t):u(x,t)>0\}=\partial\{(x,t):u(x,t)<0\}.
\end{equation}
The importance of the no-interior condition is clear in the following result; for a more precise discussion about condition (\ref{nointerior}) see \cite{bss}.
To explain it we need to recall the concept of half relaxed limits of a locally bounded family of functions $u^\varepsilon:\R^n\times(0,+\infty)\to\R$. These are defined as
\begin{equation*}\begin{array}{l}
\liminf_{*\varepsilon\to0^+}\;u^\varepsilon(x,t):=\lim_{r\to0^+}\inf\{u^\varepsilon(y,s):0<\varepsilon<r,\;(y,s)\in B(x,r)\times(t-r,t+r)\}\\
\limsup_{\varepsilon\to0^+}^*\;u^\varepsilon(x,t):=\lim_{r\to0^+}\sup\{u^\varepsilon(y,s):0<\varepsilon<r,\;(y,s)\in B(x,r)\times(t-r,t+r)\}
\end{array}\end{equation*}

\begin{thm}\label{chi viscosol}
\begin{description}
\item[(i)] The two functions $\overline\chi(x,t)=\mathds1_{D_t^+\cup\Gamma_t}(x)-\mathds1_{D_t^-}(x),$
$\underline\chi(x,t)=\mathds1_{D_t^+}(x)-\mathds1_{D_t^-\cup\Gamma_t}(x)$
are viscosity solutions of (\ref{eqhj}) (respectively the maximal subsolution and the minimal supersolution) associated respectively with the discontinuous initial data $w_o=\mathds1_{D_o^+\cup\Gamma_o}-\mathds1_{D_o^-}$ and $w_o=\mathds1_{D_o^+}-\mathds1_{D_o^-\cup\Gamma_o}$, respectively.
 \item[(ii)] Suppose that $\Gamma_o$ has an empty interior; then the Cauchy problem (\ref{eqhj}) associated with the initial data $w_o=\mathds1_{D_o^+}-\mathds1_{D_o^-}$ has a unique discontinuous solution if and only if the no-interior condition (\ref{nointerior}) holds, and this solution is given by the function
 \begin{equation}\label{eqchi}
 \chi(x,t)=\mathds1_{D_t^+}(x)-\mathds1_{D_t^-}(x).\end{equation}
\end{description}
\end{thm}
\begin{proof} We sketch this proof for the reader's convenience since even for $\alpha$ piecewise continuous it does not change much from the one in \cite{bss,soug}, given for a continuous $\alpha$. (i) The first statement of the theorem follows from the stability of viscosity solutions which holds for discontinuous equations as well, see \cite{dzs1}. To prove that the function $\underline\chi(x,t)$ is a solution of (\ref{eqhj}) associated with the initial datum $w_o=\mathds1_{D_o^+}-\mathds1_{D_o^-\cup\Gamma_o}$,
we consider the change of variables $\psi^\epsilon(r)=\tanh\big(\frac{r-\sqrt\epsilon}{\epsilon}\big)$. Since for every $\epsilon>0$ the function $\psi^\epsilon$ is strictly increasing we also have that every $U^\epsilon(x,t)=\psi^\epsilon(u(x,t))$ is a continuous viscosity solution of (\ref{eqhj}) associated with the initial datum $\psi^\epsilon(u_o)$. Moreover we can easily see that $\underline\chi^*(x,t)=\limsup^*_{\varepsilon\to0^+}U^\epsilon(x,t)$, $\underline\chi(x,t)=\underline\chi_*(x,t)=\liminf_{*\varepsilon\to0^+}U^\epsilon(x,t)$ and hence, by the stability property of viscosity sub/super-solutions, $\underline\chi$ is a discontinuous viscosity solution of (\ref{eqhj}-i).\\
(ii) If $\Gamma_o$ has empty interior and the set $\{u=0\}$ doesn't satisfy (\ref{nointerior}), by the first part of the proof we have that $\overline\chi$ and $\underline\chi$ have different semicontinuous envelopes and are both solutions of the Cauchy problem. \\
To prove the opposite implication, assume on the contrary that condition (\ref{nointerior}) holds and let $\chi$ as in (\ref{eqchi}). Then $\chi^*=\overline\chi$, $\chi_*=\underline\chi$ and so, by (i), $\chi$ is a solution of (\ref{eqhj}-i).
If $w$ is a discontinuous solution of (\ref{eqhj}) with discontinuous initial condition $w_o=\mathds1_{D_o^+}-\mathds1_{D_o^-}$, then by comparison principle, see \cite{dzs1}, $-1\leq w\leq 1$ in $\R^n\times[0,+\infty)$. Consider now a family of increasing smooth functions $\psi_n:\R\to\R$ such that $-1\leq\psi_n\leq1,\;\psi_n(r)=1$ if $r\geq0$ and $\inf_n\psi_n=-1$ in $(-\infty,0)$. By the comparison principle, we obtain that for all $n$, $w\leq w^*\leq \psi_n(u)$ for all $n$, where $u$ is the solution of (\ref{eqhj}). Thus $w=-1$ in $D_t^-$. Similarly one proves that $w=1$ in $D_t^+$ and we conclude by the no-interion condition that
$w(\cdot,t)=\underline{\chi}=\overline{\chi}$ in $D^+_t\cup D^-_t$.
\end{proof}

\begin{rem}
In the above statement, uniqueness of discontinuous solutions is meant in the sense that $u,w$ are locally bounded, $u(x,0)=w(x,0)=w_o(x)$, they are continuous on
$\{(x,0):x\in D^+_o\cup D^-_o\}$, and $u^*=w^*$, $u_*=w_*$ in $\R^n\times [0,+\infty)$.
\end{rem}

Now we can give the definition of generalized super- and subflow with prescribed normal discontinuous velocity following \cite{bdl}, (see also \cite{bs}).
\begin{defin}\label{generalized flow} A family $\open$ (resp. $\close$) of open (resp. closed) subsets of $\RRn$ is called a \emph{generalized superflow} (resp. \emph{subflow}) with normal velocity $-\alpha(x)$ if, for any $x_0\in\RRn$, $t\in(0,T)$, $r>0$, $h>0$ so that $t+h<T$ and for any smooth function $\phi:\RRn\times[0,T]\rightarrow\RR$ such that:
\begin{description}
  \item[(i)] $\partial\phi(x,s)/\partial t+\alpha^*(x)|D\phi(x,s)|<0$ (resp. $\partial\phi(x,s)/\partial t+\alpha_*(x)|D\phi(x,s)|>0$) in $B(x_0,r]\times[t,t+h]$
  \item[(ii)] $\{x\in B(x_0,r]:\phi(x,s)=0\}\neq\emptyset$, for any $s\in[t,t+h]$ and
  $$|D\phi(x,s)|\neq0\mbox{ on }\{(x,s)\in B(x_0,r]\times[t,t+h]:\phi(x,s)=0\}$$
  \item[(iii)] $\{x\in B(x_0,r]:\phi(x,t)\geq0\}\subset\Omega_t$ (resp. $\{x\in B(x_0,r]:\phi(x,t)\leq0\}\subset\mathcal{F}_t^c$),
  \item[(iv)] $\{x\in \partial B(x_0,r]:\phi(x,s)\geq0\}\subset\Omega_s$ for all $s\in[t,t+h]$ (resp. $\{x\in \partial B(x_0,r]:\phi(x,s)\leq0\}\subset\mathcal{F}_s^c$),
\end{description}
then we have
$$\{x\in B(x_0,r]:\phi(x,s)>0\}\subset\Omega_s,\quad
(\mbox{resp. }\{x\in B(x_0,r]:\phi(x,s)<0\}\subset\mathcal{F}_s^c,)$$
for every $s\in(t,t+h)$.

A family $\open$ of open subsets of $\RRn$ is called a \emph{generalized flow} with normal velocity $-\alpha(x)$ if $\open$ is a superflow and  $(\overline{\Omega}_t)_{t\in(0,T)}$ is a subflow.
\end{defin}
\begin{rem} It follows immediately by Definition \ref{generalized flow} that a family $\open$ of open subsets of $\RRn$ is a generalized superflow with normal velocity $-\alpha(x)$ if and only if $(\Omega_t^c)_{t\in(0,T)}$ is a generalized subflow with normal velocity $\alpha(x)$.
\end{rem}
The role of the super- subflows in the level set approach is described by the following statement.
\begin{thm} \label{flow-viscositysol}
\begin{description}
\item[(i)]Let $\open$ be a family of open subsets of $\RRn$ such that the set $\Omega:=\bigcup_{t\in(0,T)}\Omega_t\times\{t\}$ is open in $\RRn\times[0,T]$. Then $\open$ is a generalized superflow with normal velocity $-\alpha$ if and only if the function $\chi=\mathds{1}_{\Omega}-\mathds{1}_{\Omega^c}$ is a viscosity supersolution of (\ref{eqhj}-i)
\item[(ii)]Let $\close$ be a family of closed subsets of $\RRn$ such that the set $\mathcal{F}:=\bigcup_{t\in(0,T)}\mathcal{F}_t\times\{t\}$ is closed in $\RRn\times[0,T]$. Then $\close$ is a generalized subflow with normal velocity $-\alpha$ if and only if the function $\overline\chi=\mathds{1}_{\mathcal{F}}-\mathds{1}_{\mathcal{F}^c}$ is a viscosity subsolution of (\ref{eqhj}-i)
\end{description}
\end{thm}
\begin{proof} The argument of the proof follows with slight changes the one given in \cite{bdl}, although $\alpha$ is discontinuous, and we omit it.  
\end{proof}

We now give a result that explicitly points out the connection between the level-set approach and the definition of generalized flow given here.
\begin{cor}\label{Cor levelset-flow}
 Assume to have two families of open subsets of $\RRn$, $(\Omega^1_t)_{t\in(0,T)}$ and $(\Omega^2_t)_{t\in(0,T)}$ such that $(\Omega^1_t)_{t\in(0,T)}$ and $((\Omega^2_t)^c)_{t\in(0,T)}$ are respectively super- and subflows with normal velocity $-\alpha$ and also $\Omega_1=\cup_{t\in(0,T)}\Omega^1_t\times\{t\}$, $\Omega_2=\cup_{t\in(0,T)}\Omega^2_t\times\{t\}$ are open and disjoint. 
Define now
 $${\underline w}(x,t)=\mathds{1}_{\Omega^1}-\mathds{1}_{(\Omega^1)^c},\quad
{\overline w}(x,t)=\mathds{1}_{(\Omega^2)^c}-\mathds{1}_{\Omega^2},$$
and note that they are lower and upper semicontinuous respectively. Extend $\underline w,\;\overline w$ by semicontinuity at $t=0$ and finally define
$$\Omega^1_0=\{x\in\RRn:\chi(x,0)=1\},\qquad\Omega^2_0=\{x\in\RRn:\overline\chi(x,0)=-1\}.$$
 Suppose moreover that there exists $(\Gamma_0,D_0^+,D_0^-)\in\mathcal{E}$ such that $D^+_0\subseteq\Omega_0^1$ and $D^-_0\subseteq\Omega_0^2$. Then, if we denote with $(\Gamma_t,D_t^+,D_t^-)$ the level set evolution of $(\Gamma D_0,D_0^+,D_0^-)$, we have:\\
\indent (i) for all $t\in[0,T)$,
 $$D^+_t\subset\Omega^1_t\subset D^+_t\cup\Gamma_t,\qquad D^-_t\subset\Omega^2_t\subset D^-_t\cup\Gamma_t,$$ \indent (ii) if $\bigcup_t \Gamma_t\times\{t\}$ satisfies the no-interior condition, then for all $t\in[0,T)$, $$D^+_t=\Omega^1_t,\qquad D^-_t=\Omega^2_t.$$
\end{cor}
\begin{proof} Define ${\underline\chi}$ and $\overline \chi$ as in Theorem \ref{chi viscosol}. For the first part of the statement,
by Theorem \ref{flow-viscositysol}(i) the function $\underline w$ is a supersolution of (\ref{eqhj}), therefore by its initial condition ${\underline w}(x,0)\geq{\underline\chi}(x,0)$ and Theorem \ref{chi viscosol}(i) we get that ${\underline w}\geq{\underline\chi}$ which is minimal among supersolutions. Similarly $\overline w$ is a subsolution of (\ref{eqhj}), therefore ${\overline w}\leq {\overline\chi}$ which is maximal among subsolutions.
Comparing the definitions now the conclusion follows.

Similarly for the second part of the statement.
\end{proof}

\section{Asymptotics of reaction-diffusion equations} \label{f hyp}

We now list the main assumptions for our problem that will hold for the rest of the paper except Section 5. Most of them are technical conditions stated in the way we will need them. In the case that the nonlinearity is as in (\ref{fepsilon def}), they will follow easily from a few regularity hypotheses on the family $\{c^\gep\}_{\gep>0}$.

For the data of the Cauchy problem (\ref{asymptotic-probl}), we suppose that
$g\in C(\R^n)$, $-1\leq g\leq1$ while $f^\gep\in C^2(\RR\times\RRn)$, satisfies the following properties, where $\gamma,\rho\in(0,1)$:
\begin{equation}\label{fproperty}
\begin{split}
\left\{
\begin{array}{l}
\mbox{for any } x\in\RRn\;
f^\varepsilon(\cdot,x)\mbox{ has exactly three zeroes }-1,m_o^\gep(x),1,\;0<\rho<m_o^\varepsilon(x)<1-\rho,\\
f^\varepsilon(\cdot,x)>0\mbox{ in }(-1,{m_o^\varepsilon(x)})\cup(1,+\infty)\mbox{ and }f^\varepsilon(\cdot,x)<0\mbox{ in }(-\infty,-1)\cup({m_o^\varepsilon(x)},1),\\
\mbox{there exists a }\gamma>0\mbox{ such that }f^\varepsilon_q(q,x)\geq\gamma\mbox{ for all }q\leq-1+\gamma\mbox{ or }q\geq1-\gamma,\hbox{ and }x\in\R^n,\\
f^\varepsilon_{qq}(-1,x)<0\mbox{ and }f^\varepsilon_{qq}(1,x)>0,
\end{array}
\right.
\end{split}
\end{equation}
and also, for some $k\in[0,\frac12]$,
\begin{equation}\label{fproperty2}
\begin{split}
\left\{
\begin{array}{l}
\mbox{for every compact } K \subset\RR\mbox{ there exist constants }C=C(K)>0\\
\mbox{ such that, for all }(q,x)\in K\times\RRn,\:1\leq i,j\leq n,\\
|f^\varepsilon_{q}(q,x)|,|f^\varepsilon_{qq}(q,x)|\leq C,\;|f^\varepsilon_{x_i}(q,x)|,|f^\varepsilon_{x_iq}(q,x)|\leq\frac{C_1}{\varepsilon^k},\;|f^\varepsilon_{x_ix_j}(q,x)|\leq\frac{C_2}{\varepsilon^{2k}}.
\end{array}
\right.
\end{split}
\end{equation}
Below we denote with $\overline m(x)=\limsupstar m_o^\gep(x)$, ${\underline m}(x)=\liminfstar m_o^\gep(x)$ the upper semicontinuous and, respectively, lower semicontinuous half relaxed limits of the family $\{m_o^\gep\}_{\gep>0}$.
We also assume on $f$ that: for every compact $K_1\subset\R^n$ and $m_1>\sup_{x\in K_1}\overline m(x)$, $m_2<\inf_{x\in K_1}\underline m$,
there are two functions
\begin{equation}\label{eqestf}
\begin{array}{c}
\bar f,\;{\underline f}\in C^2(\RR\times \R^n)\mbox{ satisfying }(\ref{fproperty}),(\ref{fproperty2})\mbox{ with zeroes in }\{-1,m_1,1\},\{-1,m_2,1\}\mbox{ respectively, and}\\
{\underline f}\leq f^\gep\leq\overline f,\mbox{ for all }x\in K_1, \;q\in[-1,1],\; \gep>0\mbox{ sufficiently small.}
\end{array}\end{equation}
The typical example for the function $f^\gep$ is
\begin{equation}\label{fepsilon defa}
f^\varepsilon(q,x):=2\Big(q-\frac{c^\varepsilon(x)}{2}\Big)(q^2-1).
\end{equation}
It satisfies all the assumptions listed above with $m_o^\varepsilon(x)=c^\varepsilon(x)/2$, provided that
\begin{equation}\label{eqasscgep}
\begin{array}{ll}
c^{\varepsilon}\in C^2(\R^n), 0<\rho<c^\varepsilon(x)/2<1-\rho,\\
\abs{\partial_{x_i}c^\gep(x)}\leq\frac{C_1}{\varepsilon^k},\;\abs{\partial_{x_ix_j}^2c^\gep(x)}\leq \frac{C_2}{\varepsilon^{2k}},\;\forall x\in\R^n,\;i,j\in\{1,\dots n\},
\end{array}\end{equation}
and in (\ref{eqestf}) we can choose $\overline f(q):=2(q-m_1)(q^2-1)$, $\underline f(q):=2(q-m_2)(q^2-1)$.

Thanks to these properties of $f^\varepsilon$, as proven by Aronsson-Weinberger \cite{aw} and Fife-McLeod \cite{fml}, for all $x\in\R^n$ there is a unique pair $(q^\gep(\cdot),c^\gep(x))$, solution of the traveling wave equation
\begin{equation}\label{travelling-wave eq}
q^\varepsilon_{rr}(r,x)+c^\varepsilon(x)q^\varepsilon_r(r,x)=f^\varepsilon(q^\varepsilon(r,x),x),\qquad(r,x)\in\RR\times\RRn,
\end{equation}
subject to the following conditions
\begin{equation*}
q^\gep(-\infty,x)=-1,\;q^\gep(+\infty,x)=1,\;q^\gep(0,x)={m_o^\gep(x)}
\end{equation*}
and we have that $q^\gep_r>0$.

We will further assume that the pair $(q^\gep(\cdot),c^\gep(x))$ satisfies a series of properties.
There are $a,b>0$ such that
\begin{equation}\label{travelling wave properties a}
\inf_{x\in\R^n}q^\varepsilon(r,x)\geq1-ae^{-br}\mbox{ as } r\to+\infty,\quad
\sup_{x\in\R^n}q^\varepsilon(r,x)\leq-1+ae^{br}\mbox{ as } r\to-\infty,
\end{equation}
and moreover
\begin{equation}\label{travelling wave properties b}
\begin{array}{cc}
q^\varepsilon_r(r,x)\geq K(x,\bar r)>0,\quad\hbox{for }x\in\R^n,\;|r|\leq\bar r,\\
\sup_{(r,x)\in\RR\times\RRn}[(1+|r|)q^\varepsilon_r(r,x)+(1+|r|^2)q^\varepsilon_{rr}(r,x)]<+\infty.
\end{array}
\end{equation}
For any compact $K_1\subset\R^n$ there exist constants $M_1,M_2>0$ such that
\begin{equation}\label{travelling wave properties}
\begin{array}{l}
|Dq^\varepsilon(r,x)|,\;|Dq^\varepsilon_r(r,x)|\leq \frac{M_1}{\varepsilon^k},\;|D^2 q^\varepsilon(r,x)|\leq \frac{M_2}{\varepsilon^{2k}},\;\mbox{for all }x\in K_1,\,r\in\RR.
\end{array}\end{equation}
For instance in the case (\ref{fepsilon defa}), as well known, easy explicit calculations are possible, the traveling wave equation admits as unique solution the function
\begin{equation}
q^\varepsilon(r,x)=\tanh(r+r^\varepsilon(x)),
\end{equation}
where $r^\varepsilon(x)=\frac{1}{2}\ln\left(\frac{2+c^\varepsilon(x)}{2-c^\varepsilon(x)}\right)$ and the velocity of the traveling wave is precisely $c^\gep(x)$ of (\ref{fepsilon defa}). Some simple computations, using the properties of $c^\varepsilon$, show that for each $\varepsilon>0$, (\ref{travelling wave properties a}), (\ref{travelling wave properties b}), (\ref{travelling wave properties}) are satisfied for each $\gep>0$.

We also notice that there exists a $\bar\delta$ such that, for all $\delta\in[-\bar\delta,\bar\delta]$ the function $f^{\varepsilon,\delta}=f^\varepsilon+\delta$ satisfies similar properties to those of $f^{\varepsilon}$, (\ref{fproperty}) (\ref{fproperty2}) and (\ref{eqestf}), and it has exactly three zeroes in
$m_-^{\gep,\delta}(x)<m_o^{\gep,\delta}(x)<m_+^{\gep,\delta}(x)$, and clearly
$m_-^{\gep,\delta}(x)<-1(>-1)$, $m_+^{\gep,\delta}(x)<1(>1)$ for $\delta>0(<0)$ small enough. In particular, for each $\delta\in[-\bar\delta,\bar\delta]$, there exists a unique pair $(q^{\varepsilon,\delta}(\cdot),\;c^{\varepsilon,\delta})$ which solves the traveling wave equation
$$q^{\varepsilon,\delta}_{rr}(r,x)+c^{\varepsilon,\delta}(x)q^{\varepsilon,\delta}_r(r,x)=f^{\varepsilon,\delta}(q^{\varepsilon,\delta}(r,x),x),\qquad(r,x)\in\RR\times\RRn,$$
subject to
\begin{equation*}
q^{\gep,\delta}(-\infty,x)=m_-^{\gep,\delta}(x),\;q^{\gep,\delta}(+\infty,x)=m_+^{\gep,\delta}(x),\;q^{\gep,\delta}(0,x)=m_o^{\gep,\delta}(x)
\end{equation*}
and such that $q^{\gep,\delta}_r>0$.
The pair moreover satisfies the corresponding of (\ref{travelling wave properties a}), (\ref{travelling wave properties b}), (\ref{travelling wave properties}) and we will also suppose that there is a constant $M>0$ independent of $\gep$ such that
\begin{equation}\label{eqestdelta}
\sup_{x\in\R^n}\left[|c^\gep(x)-c^{\gep,\delta}(x)|+|1-m_+^{\gep,\delta}(x)|+|1+m_-^{\gep,\delta}(x)|\right]\leq M\delta.
\end{equation}
In the case (\ref{fepsilon defa}), one can explicitly compute
$$c^{\varepsilon,\delta}(x)=2m_o^{\gep,\delta}(x)-m_+^{\gep,\delta}(x)-m_-^{\gep,\delta}(x)$$
and therefore the estimate (\ref{eqestdelta}) is an easy consequence of an uniform estimate of the derivative $|f^\gep_q(q,x)|\geq \gamma>0$, for all $x\in \R^n$ and $q$ in a neighborhood of the three zeroes, which follows from (\ref{fproperty}).

Now for the asymptotics of the velocity of the traveling waves, we suppose that there is a smooth hypersurface $\tilde \Gamma$ that satisfies
\begin{equation}\label{c epsilon estimate}\begin{array}{c}
0<2\rho\leq{n_1(x)}< {c^\varepsilon(x)}< {n_2(x)}\leq2(1-\rho),\quad\mbox{for any }x\in\RRn,\\
c^{\varepsilon}\longrightarrow\alpha,\quad \hbox{ locally uniformly off }\tilde\Gamma,
\end{array}\end{equation}
where the functions $\alpha,n_1,n_2$ are assumed as in (\ref{alpha def}).

Again in the case (\ref{fepsilon defa}), we can explicitly choose a family of velocities $c^\varepsilon$ satisfying the assumptions above, as for instance if
\begin{equation}\label{eqcgepa}
c^\varepsilon(x)=\frac{n_1(x)}{2}\Big(1-\tanh\big(\frac{\tilde d(x)}{\varepsilon^k}\big)\Big)+\frac{n_2(x)}{2}\Big(1+\tanh\big(\frac{\tilde d(x)}{\varepsilon^k}\big)\Big),
\end{equation}
where $\tilde d\in C^2(\R^n)$ and coincides with a signed distance function from $\tilde\Gamma$ in a tubular neighborhood and observe that $\overline m=\frac{\alpha^*}2$ and $\underline m=\frac{\alpha_*}2$.

\begin{rem}\label{rem1}
It is clear that the case (\ref{fepsilon defa}) is cleaner and we only need (\ref{alpha def}), (\ref{eqasscgep}) and (\ref{c epsilon estimate}) in order to have the whole set of assumptions satisfied. Many technical assumptions may thus be avoided, in particular due to the direct relationship between the unstable equilibrium and the velocity of the approximating front provided explicitly by the traveling waves.
\end{rem}

\subsection{The abstract method}
To study the asymptotics of the solutions of singular perturbation problems for semilinear reaction-diffusion equations in $\RRn$ we follow the method explained in \cite{bs} and in \cite{bdl} and briefly recall their general idea.

In our asymptotic problem we are given a family $u^\epsilon:\RRn\times[0,T)\to\R$ of bounded regular functions, $-1\leq u^\varepsilon\leq1$, the solutions of the Cauchy problem (\ref{asymptotic-probl}), for any small parameter $\epsilon>0$. Our aim is to show that there exists a generalized flow $(\Gamma_t,\Omega^+_t,\Omega^-_t)_{t\in[0,T)}$ on $\RRn$ with a discontinuous normal velocity determined by the data of the problem such that, as $\epsilon\to0$,
$$
\begin{array}{ll}
u^\epsilon(x,t)\to 1&\mbox{if }(x,t)\in\Omega^+:=\bigcup_{t\in(0,T)}\Omega^+_t\times\{t\},\\
u^\epsilon(x,t)\to -1&\mbox{if }(x,t)\in\Omega^-:=\bigcup_{t\in(0,T)}\Omega^-_t\times\{t\},
\end{array}
$$
where $\pm1\in\RR$ are the stable equilibria of the system. We introduce two open sets
\begin{equation}\label{omega1-2}
\begin{split}
\Omega^1&=\Int\Big\{(x,t)\in\RRn\times[0,T]:\liminfstard u^\epsilon(x,t)= 1\Big\}\\
\Omega^2&=\Int\Big\{(x,t)\in\RRn\times[0,T]:\limsupstard u^\epsilon(x,t)= -1\Big\},
\end{split}
\end{equation}
and define the families $\omegaone$ and $\omegatwo$ by
\begin{equation}\label{omega1-2t}
\begin{array}{ll}
\Omega^1_t=\{x\in\R^n:(x,t)\in\Omega^1\},\quad
\Omega^2_t=\{x\in\R^n:(x,t)\in\Omega^2\},
\end{array}
\end{equation}
for all $t\in(0,T)$. Obviously $\Omega^1$ and $\Omega^2$ are open and disjoint subsets of $\RRn\times(0,T)$ and so the two step functions $\chi$ and $\overline\chi$, defined as
\begin{equation}\label{chi def}
\chi(x,t)=\mathds{1}_{\Omega^1}-\mathds{1}_{(\Omega^1)^c},\quad
\overline\chi(x,t)=\mathds{1}_{(\Omega^2)^c}-\mathds{1}_{\Omega^2}
\end{equation}
are respectively lower and upper semicontinuous on $\RRn\times(0,T)$. Also notice that $\Omega^1=\cup_{t\in(0,T)}\Omega^1_t\times\{t\}$, $\Omega^2=\cup_{t\in(0,T)}\Omega^2_t\times\{t\}$. Finally we extend $\chi$, $\overline\chi$ by lower and upper semicontinuity to the whole of $\RRn\times[0,T]$. For simplicity of notation we still call $\chi$ and $\overline\chi$ these extensions.

To analyze the asymptotics for our functions $u^\epsilon$ we follow three steps.\\
1. \emph{Initialization}: we define the traces $\Omega^1_0$ and $\Omega^2_0$ of $\Omega^1$ and $\Omega^2$ for $t=0$ as
\begin{equation}\label{omega0}
\Omega^1_0=\{x\in\RRn:\chi(x,0)=1\},\qquad\Omega^2_0=\{x\in\RRn:\overline\chi(x,0)=-1\}.
\end{equation}
\emph{2. Propagation}: we show that $\omegaone$ and $((\Omega^2_t)^c)_{t\in(0,T)}$ are respectively super and sub-flows with normal velocity $-\alpha$, where $\alpha$ is defined in (\ref{c epsilon estimate}).\\
\emph{3. Conclusion}: we conclude our asymptotics by applying Corollary \ref{Cor levelset-flow} to $(\Omega^1_t)_{t\in[0,T)}$ and $((\Omega^2_t)^c)_{t\in[0,T)}$.

\subsection{The asymptotic problem}

The front associated with the asymptotics of (\ref{asymptotic-probl}) evolves according to the geometric pde (\ref{eqhj}-i) as we claim in the following theorem.
\begin{thm}\label{thm} Assume (\ref{alpha def}), (\ref{fproperty}), (\ref{fproperty2}), (\ref{eqestf}), (\ref{travelling wave properties a}), (\ref{travelling wave properties b}), (\ref{travelling wave properties}), (\ref{eqestdelta}), (\ref{c epsilon estimate}). Let $u^\varepsilon$ be the unique smooth solution of (\ref{asymptotic-probl}), where $g:\RRn\to[-1,1]$ is a continuous function such that the sets $\Gamma_o=\{x:\underline m(x)\leq g(x)\leq\overline m(x)\}$, $\Omega_o^+=\{x:g(x)>\overline m(x)\}$, $\Omega_o^-=\{x:g(x)<\underline m(x)\}$ are nonempty and $(\Gamma_o,\Omega_o^+,\Omega_o^-)\in{\mathcal E}$. Then
$$u^\varepsilon(x,t)\longrightarrow
\left\{\begin{array}{rll}
1,& &\hbox{if }u(x,t)>0,\\
-1,& &\hbox{if }u(x,t)<0,
\end{array}
\right.
$$
locally uniformly as $\varepsilon\to0$,
where $u$ is the unique viscosity solution of
\begin{equation}\label{PCfront}
\left\{\begin{array}{ll}
u_t(x,t)+\alpha(x)|Du(x,t)|=0\mbox{ in }\RRn\times(0,+\infty),\\
u(x,0)=d_o(x),
\end{array}\right.
\end{equation}
and $d_o$ is the signed distance to $\Gamma_o$ which is positive in $\Omega^+_o$ and negative in $\Omega_o^-$. If in addition the no-interior condition (\ref{nointerior}) for the set $\{u=0\}$ holds, then, as $\varepsilon\to0$,
$$u^\varepsilon(x,t)\longrightarrow
\left\{\begin{array}{rll}
1& &\{u>0\},\\
-1& &\overline{\{u>0\}}^c,
\end{array}
\right.
$$
 locally uniformly.
\end{thm}
\begin{rem}
The results of the theorem are more elegant in the case that the initialized front $\Gamma_o$  has empty interior. In the open sets where $\overline m=\underline m=m$, then the family $\{m_o^\gep\}$ converges locally uniformly, and $\Gamma_o$ is determined by the equation $g=m$. If this is not the case, $\Gamma_o$ may contain relatively open subsets of $\{x:\overline m(x)>\underline m(x)\}$. Notice  also that in the case (\ref{fepsilon defa}) then $\overline m=\frac{\alpha^*}2$ and $\underline m=\frac{\alpha_*}2$, therefore even in that case it is preferable to have a set of discontinuities of $\alpha$ with empty interior.
\end{rem}
\begin{proof} The proof will take up the rest of the section and will be divided into a series of statements. Following the abstract method described in the previous section we define two families of open sets of $\RRn$, $\omegaone$ and $\omegatwo$ as in (\ref{omega1-2}), (\ref{omega1-2t}) and two further sets $\Omega^1_0$, $\Omega^2_0$ as in ($\ref{omega0}$). We recall that by maximum principle $-1\leq u^\gep\leq1$.

\emph{First step: initialization.} We want to show that,
$$\Omega_0^+=\{d_o>0\}\subseteq\Omega_0^1,\qquad\Omega_0^-=\{d_o<0\}\subseteq\Omega_0^2.$$
Since the proofs of these two inclusions are similar we only show the first one.
Consider $\hat x\in\{d_o>0\}$, then we have that $g(\hat x)>\overline m(\hat x)$ and so, by the continuity of $g$, upper semicontinuity and definition of $\overline m$, we can find an $r,\sigma>0$ such that
$$g(x)\geq\sup_{B(\hat x,r)}\overline m+\sigma\geq m_o^\gep(x)+\frac34\sigma,$$
for all $x\in B(\hat x,r)$ and $\gep$ sufficiently small. This means that
\begin{equation}\label{lower estimate uepsilon}
u^\varepsilon(x,0)=g(x)\geq\Big(\sup_{B(\hat x,r)}\overline m+\sigma\Big)\mathds1_{B(\hat x,r)}(x)-\mathds1_{B(\hat x,r)^c}(x).
\end{equation}
Now we introduce the function $\Phi:\RRn\times[0,T]\to\RR$ defined by
\begin{equation}\label{Phi def}
\Phi(x,t)=r^2-|x-\hat x|^2-Ct,
\end{equation}
with $C>0$ a constant that will be chosen later. We denote by $d(\cdot,t)$ the signed distance to the set $\{\Phi(\cdot,t)=0\}$ defined in such a way to have the same sign of $\Phi$. Explicitly $d(x,t)=\sqrt{(r^2-Ct)^+}-|x-\hat x|$. Note in particular that $d(x,0)\geq\beta(>0)$ if and only if $x\in B(\hat x,r-\beta]$.

To prove the first step we need the two following lemmas.
\begin{lem}\label{initialization first lemma}Under the assumptions of Theorem \ref{thm} we have that for any $\beta>0$ there exist $\tau=\tau(\beta)>0$ and $\bar\varepsilon=\bar\varepsilon(\beta)$ such that, for all $0<\varepsilon\leq\bar\varepsilon$, we have
$$u^\varepsilon(x,t_\varepsilon)\geq(1-\beta)\mathds1_{\{d(\cdot,0)\geq\beta\}}(x)-\mathds1_{\{d(\cdot,0)<\beta\}}(x),\quad x\in\RRn,$$
where $t_\varepsilon=\tau\varepsilon$ and $d(x,t)=\sqrt{(r^2-Ct)^+}-|x-\hat x|$.
\end{lem}
\begin{lem}\label{initialization second lemma}There exist $\bar h=\bar h(r,\hat x)>0$, $\bar\beta=\bar\beta(r,\hat x)$ independent of $\gep$ such that if $\beta\leq\bar\beta$ and $\varepsilon\leq\bar\varepsilon(\beta)$, then there is a subsolution $\omega^{\varepsilon,\beta}$ of (\ref{asymptotic-probl}-i) in $\RRn\times(0,\bar h)$ that satisfies
$$\omega^{\varepsilon,\beta}(x,0)\leq(1-\beta)\mathds1_{\{d(\cdot,0)\geq\beta\}}(x)-\mathds1_{\{d(\cdot,0)<\beta\}}(x),\quad x\in\RRn.$$
If  moreover $(x,t)\in B(\hat x,r)\times(0,\bar h)$ and $d(x,t)>3\beta$, then
$$\liminfstard\omega^{\varepsilon,\beta}(x,t)\geq1-3\beta.$$
\end{lem}

Before proving Lemmas \ref{initialization first lemma} and \ref{initialization second lemma} we give the short conclusion of the first step which follows \cite{bdl}. To do this, we first notice that, combining these two Lemmas, we get the existence of a viscosity subsolution $\omega^{\varepsilon,\beta}$ of (\ref{asymptotic-probl}-i) in $\RRn\times(0,\bar h)$ such that
$$\omega^{\varepsilon,\beta}(x,0)\leq u^\varepsilon(x,t_\varepsilon),\quad\mbox{for all }x\in\RRn,$$
and so, by the maximum principle,
$$\omega^{\varepsilon,\beta}(x,s)\leq u^\varepsilon(x,s+t_\varepsilon),\quad\mbox{for all }(x,s)\in\RRn\times[0,\bar h].$$
Therefore, using the second part of Lemma \ref{initialization second lemma}, we get that for all $(x,s)\in B(\hat x,r)\times(0,\bar h)$, $d(x,s)>3\beta$,
$$\liminfstard u^\varepsilon(x,s)\geq1-3\beta.$$
Since $\beta$ is arbitrary and does not depend on $\bar h$ we can send it to zero in order to obtain that, for all $(x,s)\in B(\hat x,r)\times(0,\bar h)$, $d(x,s)>0$,
$$\liminfstard u^\varepsilon(x,s)\geq1,$$
i.e. $x\in\Omega^1_s$ by definition. 

Moreover, by definition of $d$, it follows that there exist $\bar\eta<r$, $\bar t<\bar h$ so that $B(\hat x,\bar\eta)\subset\{d(\cdot,t)>0\}$ for any $0<t<\bar t$. This implies that $B(\hat x,\bar\eta)\subset\Omega^1_t$ for any $0<t<\bar t$ and therefore $\chi(\hat x,0)=1$ and $\hat x\in \Omega_0^1$.

\begin{proof}[Proof of Lemma $\ref{initialization first lemma}$] For the proof of this lemma we follow the ideas of Chen \cite {xce,xc}, based on the fact that for $\gep$ small in the reaction diffusion equation the diffusion term is negligible for short time, and of Barles-Da Lio \cite{bdl}.
The lemma is a local short time generation of the interface. The corresponding proof in \cite{xce} is more precise since there the time needed to generate the interface is precisely determined. Let $\beta>0$ be fixed. Due to the maximum principle we just need to show that $u^\gep(x,t_\gep)\geq1-\beta$ if $d(x,0)\geq\beta$. 
\\
1. 
We denote by $\chi=\chi(\tau,\xi;x)\in C^2([0,+\infty)\times\RR\times\RRn)$ the solution of
\begin{equation}\label{ordinary equation chi}
\left\{
\begin{array}{l}
\dot\chi(\tau,\xi;x)+f^\varepsilon(\chi(\tau,\xi;x),x)=0,\quad \tau>0,\\
\chi(0,\xi;x)=\xi.
\end{array}
\right.
\end{equation}
It is then simple to see, by the properties of ordinary differential equations, that $\chi$ satisfies the following properties
\begin{equation}\tag{$\chi1$}\label{property one chi}
\chi_\xi(\tau,\xi;x)>0,\quad\mbox{ in }[0,+\infty)\times\RR\times\RRn,
\end{equation}
and there exists $\tau_o=\tau_o(\beta)>0$ such that, for all $\tau\geq\tau_o$
\begin{equation}\tag{$\chi2$}\label{property two chi}
\begin{array}{l}
\chi(\tau,\xi;x)\geq1-\beta,\quad\forall\,\xi\geq\sup_{B(\hat x,r)}\overline m+\frac\sigma2.
\end{array}
\end{equation}
(Regarding the proof of the estimate in (\ref{property two chi}), which is independent of $\varepsilon$ and $x$, we just notice that we can choose a cubic-like function $\overline f$ as in (\ref{eqestf}) with $K=B(\hat x,r],\;m_1=\sup_{B(\hat x,r)}\overline m+\frac\sigma4$ such that
$$\overline f(q)\geq f^\gep(q,x),$$
for all $x\in B(\hat x,r)$, $q\in[-1,1]$, and $\gep$ sufficiently small.)\\
Moreover, since for any $C>1$ we have that $\chi(\tau,\xi,x)\in[-C,C]$ for all $\xi\in[-C,C]$, $\tau\geq0$, $x\in\RRn$, it also holds that
for any $C>1$, $\tau>0$ there exists a constant $M_{C,\tau}>0$ such that
\begin{equation}\tag{$\chi3$}\label{property three chi}
\begin{array}{l}
|\chi_{\xi\xi}(\tau,\xi;x)|\leq M_{C,\tau}\chi_\xi(\tau,\xi;x),\quad
|\chi_{x_i}(\tau,\xi;x)|,\leq \frac{M_{C,\tau}}{\gep^k}\\
|\chi_{\xi x_i}(\tau,\xi;x)|\leq \frac{M_{C,\tau}}{\gep^k}\chi_\xi(\tau,\xi;x),\quad
|\chi_{x_ix_i}(\tau,\xi;x)|\leq \frac{M_{C,\tau}}{\gep^{2k}}\chi_\xi(\tau,\xi;x),
\end{array}
\end{equation}
for any $\xi\in[-C,C],\,x\in\RRn,\,i\in\{1,2,\cdots,n\}$ and $\gep$ small enough.

2. Let $\psi$ be a nondecreasing smooth function in $\RR$ such that
$$\psi(z)=\left\{\begin{array}{ll}
-1&\mbox{if }z\leq 0,\\
\sup_{B(\hat x,r)}\overline m+\sigma&\mbox{if }z\geq\beta\wedge\frac{\sigma}{2}.
\end{array}\right.
$$
We can define a function $\underline u^\varepsilon$ in $\RRn\times[0,T]$ as
$$\underline u^\varepsilon(x,t)=\chi\Big(\frac{t}{\varepsilon},\psi(d(x,0))-Kt,x\Big),$$
for $K$ a constant to be decided later.
Thanks to a computation similar to those in \cite{bs} one can prove that, if $K$ is large enough, $\underline u^\varepsilon$ is a subsolution of (\ref{asymptotic-probl}-i) in $\RRn\times(0,\tau_o\varepsilon)$, with $\tau_o$ as in (\ref{property two chi}). In fact, since $\chi$ satisfies (\ref{ordinary equation chi}) and $\psi'$ has compact support, we obtain
\begin{equation}\label{computation}
\begin{split}
\underline u_t^\varepsilon-\varepsilon\Delta\underline u^\varepsilon+\frac{f^\varepsilon(\underline u^\varepsilon,x)}{\varepsilon}&=\frac{\dot\chi}{\varepsilon}-K\chi_\xi-\varepsilon\Big[\chi_{\xi\xi}\bigabs{{\psi}' Dd(x,0)}^2\\
&+\chi_\xi\Big(\psi''+\psi'\Delta d(x,0)\Big)+\Delta\chi
+2D\chi_\xi\cdot\Big(\psi'Dd(x,0)\Big)\Big]+\frac{f^\varepsilon(\chi,x)}{\varepsilon}\\
&\leq-K\chi_\xi+\varepsilon[M_1\abs{\chi_{\xi\xi}}+M_2\chi_\xi+\abs{\Delta\chi}+M_3\abs{D\chi_\xi}].
\end{split}
\end{equation}
Now we want to use properties (\ref{property one chi}) and (\ref{property three chi}) in order to get an estimate for the terms $\abs{\chi_{\xi\xi}}$, $\abs{D\chi_\xi}$, $\abs{\Delta\chi}$. Indeed since $\psi(d(x,0))\in I=[-1,1+\sigma]$ for all $x\in\RRn$, by evaluating (\ref{computation}) at a point of $\RRn\times(0,\tau_o\varepsilon)$ we obtain
\begin{equation*}
\underline u_t^\varepsilon-\varepsilon\Delta\underline u^\varepsilon+\frac{f^\varepsilon(\underline u^\varepsilon,x)}{\varepsilon}
\leq-\chi_\xi\left(K-\varepsilon M_2-\varepsilon M_{2,\tau_o}\left(M_1+\frac{M_3}{\gep^k}+\frac1{\gep^{2k}}\right)\right)\leq0,
\end{equation*}
for $K$ large enough. Moreover by definition of $d$,
\begin{equation*}
\begin{split}
\underline u^\varepsilon(x,0)=\psi(d(x,0))&\leq \Big(\sup_{B(\hat x,r)}\overline m+\sigma\Big)\mathds1_{\{d(x,0)>0\}}(x)-\mathds1_{\{d(x,0)\leq0\}}(x)\\
                             &= \Big(\sup_{B(\hat x,r)}\overline m+\sigma\Big)\mathds1_{B(\hat x,r)}(x)-\mathds1_{B(\hat x,r)^c}.
\end{split}
\end{equation*}
Therefore combining the last inequality with (\ref{lower estimate uepsilon}) we get
$$\underline u^\varepsilon(x,0)\leq u^\varepsilon(x,0),\mbox{ for all $x$ in }\RRn.$$
Thus, by the maximum principle,
$$\underline u^\varepsilon(x,t)\leq u^\varepsilon(x,t)\;\mbox{in }\RRn\times[0,\tau_o\varepsilon].$$
Now if we evaluate the last inequality for $x\in\{d(\cdot,0)\geq\beta\wedge\sigma/2\}$ and $t=t_\varepsilon=\tau_o\varepsilon$, we get
$$u^\varepsilon(x,t_\varepsilon)\geq\chi\big(\tau_o,\sup_{B(\hat x,r)}\overline m+\sigma-K\tau_o\varepsilon,x\big)
\geq \chi\big(\tau_o,\sup_{B(\hat x,r)}\overline m+\frac\sigma2,x\big),$$
for $\varepsilon\leq \frac\sigma{2K\tau_o}$.
Therefore by (\ref{property two chi}) and we obtain
$$u^\varepsilon(x,t_\varepsilon)\geq1-\beta,$$
for all $x\in\{d(\cdot,0)\geq\beta\}$.
\end{proof}

\begin{proof}[Proof of Lemma $\ref{initialization second lemma}$] The proof follows with some modifications the ideas in \cite{bdl} and \cite{bs}. First of all we consider the smooth function $\Phi$ defined in (\ref{Phi def}) where now $C$ is fixed and satisfies
\begin{equation}\label{C assumption}
C\geq 8r.
\end{equation}

Since $D \Phi(x,t)\neq0$ if $\Phi(x,t)=0$, there exist $\gamma,\,\bar h>0$ such that $\bar h<r^2/C$, $d$ is smooth in the set $Q_{\gamma,\bar h}=\{(x,t):\abs{(d(x,t))}\leq\gamma,\,|x-\hat x|\geq\gamma,\;0\leq t\leq\bar h\}$, and $D\Phi(x,t)\neq 0$ in $Q_{\gamma,\bar h}$. Now we construct a subsolution by steps.

1. We first define a smooth function $v^\varepsilon$ in $Q_{\gamma,\bar h}$ as
$$v^{\varepsilon}(x,t)=q^{\varepsilon,\delta}\Big(\frac{d(x,t)-2\beta}{\varepsilon},x\Big)-2\beta,$$
with $\delta\in[0,\bar\delta]$ to be chosen later.
Using the definition of $d$, the assumption (\ref{C assumption}) on $C$ and the properties (\ref{travelling wave properties}) satisfied by $\qued$ we can see that in $Q_{\gamma,\bar h}$,
\begin{equation*}
\begin{split}
v_t^{\varepsilon}-\varepsilon\Delta v^{\varepsilon}+\frac{f^\varepsilon(v^\varepsilon,x)}{\varepsilon}&=\frac{\qued_r d_t}{\varepsilon}-\frac{\qued_{rr}}{\varepsilon}-2D\qued_r\cdot Dd-q_r\Delta d-\varepsilon\Delta\qued+\frac{f^\varepsilon(\qued-2\beta,x)}{\varepsilon}\\
&\leq\frac{\qued_r}{\varepsilon}\big(\frac{-C}{2\sqrt{r^2-Ct}}+c^{\varepsilon,\delta}(x)+\gep\frac{n-1}{\abs{x-\hat x}}\big)
-\frac{\delta}{\varepsilon}-2D\qued_r\cdot Dd-\varepsilon\Delta\qued\\
&\qquad\qquad\qquad\qquad\qquad\qquad\qquad\qquad-\frac{2\beta f^\gep_q(\qued,x)}{\gep}+\frac{2\beta^2\norm{f_{qq}^\gep}_\infty}{\gep}\\
&\leq\frac{1}{\gep}\Big[-\qued_r-2\beta f^\gep_q(\qued,x)+2\beta^2\norm{f_{qq}^\gep}_\infty\Big]+\left[-\frac{\delta}{\gep}+2\abs{D\qued_r}+\gep\abs{\Delta\qued}\right],
\end{split}
\end{equation*}
for $\gep$ and $|\delta|$ small enough.
Since for any $x\in\RRn$, $\delta\in[0,\bar\delta]$,
$$\qued(\cdot,x)\in[m_-^{\gep,\delta}(x),m_+^{\gep,\delta}(x)]\subseteq[-1-\delta,1+\delta],$$
here and below the $L^\infty$ norm of the derivatives of $f^\varepsilon$ are taken for its first argument $q$ in the compact set $[-1-\bar\delta,1+\bar\delta]$.
To prove that $v^\gep$ is a subsolution of (\ref{asymptotic-probl}-i) it remains to see that the right hand side of the last inequality above is non positive.
For the right bracket we compute
$$-\frac{\delta}{\gep}+2\abs{D\qued_r}+\gep\abs{\Delta\qued}\leq -\frac{\delta}{\varepsilon}+\frac{2M_1}{\varepsilon^k}+\gep\frac{M_2}{\varepsilon^{2k}}
\leq -\frac{\delta}{2\varepsilon}
$$
when $\delta>0$ is fixed and $\varepsilon$ is small enough. For the left bracket, we combine (\ref{fproperty}) and (\ref{eqestdelta}), $f_q^\gep(m_\pm^{\gep,\delta}(x),x)\geq \gamma>0$ and $\qued(r,x)\to m_\pm^{\gep,\delta}(x)$ if $r\to\pm\infty$ exponentially fast, uniformly for $x\in\R^n$.
This means that we may suppose that there exists an $\bar r>0$ such that
$$f_q^\gep(\qued(r,x),x)\geq\frac{\gamma}{2},\quad\mbox{for any }|r|\geq\bar r,$$
and we can choose $\beta$ small enough, independent of $\gep$, $\delta$, in order to get
$$\beta\norm{f_{qq}^\gep}_\infty=\beta\sup\{\abs{f_{qq}^\gep(q,x)}:(q,x)\in[-1-\bar\delta,1+\bar\delta]\times\RRn\}\leq\frac{\gamma}{2}.$$
Thus we consider two cases. If ${|d(x,t)-2\beta|}\geq\gep\bar r$, we have that
$$v_t^{\gep}-\gep\Delta v^{\gep}+\frac{f^\gep(v^\gep,x)}{\gep}\leq-\frac{\qued_r}{\gep}-\frac{\delta}{2\gep}<0$$
for $\gep$ small enough.
If, on the other hand, ${|d(x,t)-2\beta|}< \gep\bar r$ and we denote with $K$ a strictly positive constant (which depends on $\bar r$) so that  $\qued_r(r,x)\geq K>0$ for any $|r|\leq\bar r$, $x\in\RRn$, we get that, for $\beta$ small compared to $K$,
 \begin{equation*}
v_t^{\gep}-\gep\Delta v^{\gep}+\frac{f^\gep(v^\gep,x)}{\gep}\leq\frac{1}{\gep}(-K+2\beta(\norm{f_q^\gep}_\infty +2\beta\norm{f_{qq}^\gep}_\infty))-\frac{\delta}{2\gep}<0.\end{equation*}

2. We now define in $\{(x,t)\in\RRn\times[0,\bar h]:d(x,t)\leq\gamma\}$,
$$\bar v^\varepsilon(x,t)=
\left\{\begin{array}{lll}
\sup(v^\varepsilon(x,t),-1)&\mbox{if}&-\gamma<d(x,t)\leq\gamma,\\
-1&\mbox{if}&d(x,t)\leq-\gamma.
\end{array}
\right.$$
By a similar reasoning to that of Lemma 4.4 in \cite{bs} one easily proves that $\bar v^\varepsilon$ is a continuous viscosity subsolution of (\ref{asymptotic-probl}-i) in $\{(x,t)\in\RRn\times[0,\bar h]:d(x,t)\leq\gamma\}$, for $\gep$ sufficiently small.

3. We finally define our function $\omega^{\varepsilon,\beta}:\RRn\times[0,\bar h]\to\RR$ as
$$
\omega^{\varepsilon,\beta}(x,t)=
\left\{\begin{array}{ll}
\psi(d(x,t))\bar v^{\varepsilon}(x,t)+(1-\psi(d(x,t)))(1-\beta)&\mbox{if }d(x,t)<\gamma,\\
1-\beta&\mbox{if }d(x,t)\geq\gamma,
\end{array}
\right.
$$
where $\psi:\RR\to\RR$ is a smooth function such that $\psi'\leq0$ in $\RR$, $\psi=1$ in $(-\infty,\gamma/2]$, $0<\psi<1$ in $(\gamma/2,3\gamma/4)$ and $\psi=0$ in $[3\gamma/4,+\infty)$. The only subset of $\RRn\times(0,\bar h)$ in which we have to check that $\omegaeb$ is a subsolution is $\{(x,t)\in\RRn\times(0,\bar h):\gamma/2\leq d(x,t)\leq3\gamma/4\}$. Since $\abs{Dd}=1$
\begin{equation}\label{omegaeb inequal}
\begin{split}
\omegaeb_t-\varepsilon\Delta\omegaeb+\frac{f^\varepsilon(\omegaeb,x)}{\varepsilon}=&\psi(\bar v^\varepsilon_t-\varepsilon\Delta\bar v^\varepsilon)-2\varepsilon\psi' Dd\cdot D\bar v^\varepsilon\\
&+(\psi'd_t-\gep\psi'\Delta d-\varepsilon\psi'')(\bar v^\varepsilon-(1-\beta))
+\frac{f^\varepsilon(\omegaeb,x)}{\varepsilon}.
\end{split}\end{equation}
If we take $2\beta<\gamma/4$
$$\begin{array}{ll}
v^\varepsilon(x,t)&\geq\qued\Big(\frac{\gamma}{4\gep},x\Big)-2\beta\\
&\geq m_+^{\gep,\delta}(x)-ae^{-\frac{b\gamma}{4\varepsilon}}-2\beta\geq1-M\delta-ae^{-\frac{b\gamma}{4\varepsilon}}-2\beta
\end{array}$$ 
and so for $\gep,\beta,\delta$ small $\bar v^\varepsilon(x,t)=v^\varepsilon(x,t)$ and $\bar v^\varepsilon(x,t)-(1-\beta)\leq-\beta$. Moreover, since $f^\varepsilon_{qq}(1,x)>0$, $f^\varepsilon(\omegaeb,x)\leq\psi f^\varepsilon(v^\varepsilon,x)+(1-\psi)f^\varepsilon(1-\beta,x)$, (\ref{omegaeb inequal}) becomes
\begin{equation*}
\begin{split}
\omegaeb_t-\varepsilon\Delta\omegaeb+\frac{f^\varepsilon(\omegaeb,x)}{\varepsilon}\leq&-\psi\frac{\delta}{2\gep}-2\psi' q^\gep_r+2 \varepsilon ^{1-k}M_1\\
&+\psi'd_t(v^\varepsilon-(1-\beta))+(1-\psi)\frac{f^\varepsilon(1-\beta,x)}{\varepsilon}+O(\varepsilon)\\
\leq&-\frac1\gep\left(\psi\frac\delta2+(1-\psi)(-f^\gep(1-\beta,x))\right)+\tilde M_3+o_\varepsilon(1)\leq0,
\end{split}\end{equation*}
for $\varepsilon$ small enough. To get the last inequality, we also used the fact that $d_t\leq0$ and $\sup_{x\in\R^n}f^\varepsilon(1-\beta,x)<0$ for $\beta$ small enough.

4. Now we observe that, if $d(x,t)<\beta$, then $v^\varepsilon(x,t)\leq q^{\gep,\delta}(-\frac{\beta}\gep,x)-2\beta\leq m_-^{\gep,\delta}(x)+ae^{-\frac{b\beta}\gep}-2\beta\leq m_-^{\gep,\delta}(x)\leq-1$ for $\gep$ small enough (and $\beta$ fixed).
This means that, for $\varepsilon$ small enough
$$v^\varepsilon(x,t)\leq(1-\beta)\mathds1_{\{d\geq\beta\}}(x,t)-\mathds1_{\{d<\beta\}}(x,t).$$
By definition of $\bar v^\varepsilon$ and of $\omegaeb$ the last inequality still holds for $\bar v^\varepsilon$ and $\omegaeb$ (we just point out that if $d(x,t)\geq\beta$ then $\omegaeb(x,t)$ is equal to $1-\beta$ or to a convex linear combination of elements of $(-\infty,1-\beta]$). If we consider $t=0$ we have proved the second part of our Lemma.

5. Finally we just remark that, with a reasoning similar to the one in point 4. one can prove that if $(x,t)\in B(\hat x,r)\times(0,\bar h)$ and $d(x,t)>3\beta$, then
$$v^\gep(x,t)\geq q^{\gep,\delta}(\frac\beta\gep,x)-2\beta\geq1-ae^{-\frac{b\beta}\gep}-2\beta-M\delta.$$
Hence $\liminfstar\omega^{\varepsilon,\beta}(x,t)\geq1-3\beta,$
for $\beta\geq M\delta$.
\end{proof}

\emph{Second step: propagation.} In this step we show that $\omegaone$ and $((\Omega^2_t)^c)_{t\in(0,T)}$ are respectively super and sub-flows with normal velocity $-\alpha$. Since the two proofs are similar we only show that $\omegaone$ is a superflow. One of the new difficulties here is due to the fact that the flow has a discontinuous velocity and we will need to approximate the definition of super- and subflow by using continuous velocities.
We do that by means of the same smooth functions $c^\varepsilon$ appearing in the problem. We consider the following modified families of continuous functions and define:
$$\overline c{^\gep}(x):=\eta^\gep(x)n_2(x)+(1-\eta^\gep(x))c^\gep(x),\quad
\underline c{^\gep}(x):=\xi^\gep(x)n_1(x)+(1-\xi^\gep(x))c^\gep(x),$$
where $\eta^\gep,\xi^\gep\in C^2(\R^n)$, $\eta^\gep(x),\;\xi^\gep(x)\in[0,1]$,
$$\eta^\gep(x):=
\left\{
\begin{array}{ll}
1&\mbox{if } \tilde d(x)\geq -\gep\\
0&\mbox{if } \tilde d(x)\leq-2\gep
\end{array}
\right.;\quad
\xi^\gep(x):=
\left\{
\begin{array}{ll}
1&\mbox{if } \tilde d(x)\leq \gep\\
0&\mbox{if } \tilde d(x)\geq2\gep
\end{array}
\right..
$$
Notice that
\begin{equation*}
n_1\leq \underline c^\gep\leq c^\gep\leq \overline c^\gep\leq n_2,\quad \underline c^\gep\leq \alpha_*\leq\alpha^*\leq \overline c^\gep
\end{equation*}
and $\limsupstar \overline c^\gep(x)=\alpha^*(x)$, $\liminfstar \underline c^\gep(x)=\alpha_*(x)$
We denote below as ${\overline{\mathcal F}}=\{\bar c^\gep,\,\varepsilon>0\}$, ${\underline{\mathcal F}}=\{\underline c^\gep,\,\varepsilon>0\}$.

\begin{prop}\label{equivalent def of generalized flow}
\begin{description}
\item[(i)] A family $\open$ of open subsets of $\RRn$, such that the set $\Omega:=\bigcup_{t\in(0,T)}\Omega_t\times\{t\}$ is open in $\RRn\times[0,T]$, is a \emph{generalized superflow} with normal velocity $-\alpha$ if and only if
it is a generalized superflow with normal velocity $- \overline c\in C(\R^n)$, for all $\bar c\in\overline{\mathcal F}$;
\item[(ii)] A family $\close$ of close subsets of $\RRn$ such that the set $\mathcal{F}:=\bigcup_{t\in(0,T)}\mathcal{F}_t\times\{t\}$ is closed in $\RRn\times[0,T]$ is a \emph{generalized subflow} with normal velocity $-\alpha$ if and only if
it is a generalized subflow with normal velocity $-\underline c$, for all $\underline c\in\underline{\mathcal F}$.
\end{description}
\end{prop}
\begin{proof}(i)  In view of Theorem \ref{flow-viscositysol}, in order to prove this statement we have to prove that the function $\chi=\mathds{1}_{\Omega}-\mathds{1}_{\Omega^c}$ is a viscosity supersolution of (\ref{eqhj}-i) if and only if it is a viscosity supersolution of
\begin{equation}\label{PCi}
\chi_t(x,t)+{\overline c}^\gep(x)|D\chi(x,t)|=0,\quad(x,t)\in\R^n\times(0,T),
\end{equation}
for all $\gep>0$.
We start assuming that
for every continuous function $\overline c^\gep$, $\chi$ is a viscosity supersolution of (\ref{PCi}).
The conclusion follows from the stability of viscosity supersolutions and the fact that $\alpha^\star=\limsup^*_{\gep\to0^+}\overline c^\gep$. Therefore $\chi$ is a supersolution also of (\ref{eqhj}-i). Since ${\overline c}^\gep\geq \alpha^*$, the other implication is trivial.
\\
(ii) The proof concerning the subflow is similar and we omit it.
\end{proof}
Next we want to show that $\omegaone$ is a superflow with normal velocity $-\overline c$, for any $\overline c\in\overline{\mathcal F}$.
\begin{prop}\label{prop propagation} Let $\overline c\in\overline{\mathcal F}$ be fixed and let $x_0\in\RRn$, $t\in(0,T)$, $r>0$, $h>0$ so that $t+h<T$. Suppose that $\phi:\RRn\times[0,T]\to\RR$ be a smooth function that satisfies (i)--(iv) in Definition \ref{generalized flow} with $\Omega^1_s$ substituting $\Omega_s$.
Then, for every $s\in(t,t+h)$,
$$\{x\in B(x_0,r]:\phi(x,s)>0\}\subset\Omega_s^1.$$
\end{prop}
\begin{proof}
Using the assumptions and the definition of $(\Omega^1_t)_{t\in(0,T)}$ we need to prove that for all $x\in B(x_0,r)$, $s\in(t,t+h)$ such that $\phi(x,s)>0$, then we have
$$\liminfstard u^\varepsilon(y,\tau)\geq1$$
for $(y,\tau)$ in a neighborhood of $(x,s)$. 
By (i), let $\tilde C>0$ be such that
$$\phi_t(x,s)+\overline c(x)|D\phi(x,s)|\leq-\tilde C<0,\quad\hbox{ for all }(x,s)\in B(x_0,r]\times[t,t+h].$$
The proof proceeds like the one of the first step with the difference that here we have to construct a subsolution of (\ref{asymptotic-probl}-i) only in the ball $B(x_0,r)$ and not in the whole space $\RRn$. We will need to use an extra boundary condition coming from (iv). In fact to prove this result it is enough to prove the following lemma which plays the role of Lemma \ref{initialization second lemma} in the first step. We denote below with $d(\cdot,s)$ the signed distance function to the set $\{\phi(\cdot,s)=0\}$ which has the same sign of $\phi$.
\begin{lem}\label{lemma propagation} Let the assumptions of Proposition \ref{prop propagation} hold true.
There exists $\bar\beta$ small enough such that, if $\beta\leq\bar\beta$ and $\varepsilon\leq\bar\varepsilon(\beta)$ then there is a viscosity subsolution $\omega^{\varepsilon,\beta}$ of (\ref{asymptotic-probl}-i) in $B(x_0,r)\times(t,t+h)$ that satisfies,

 1. $\omega^{\varepsilon,\beta}(x,t)\leq(1-\beta)\mathds1_{\{d(\cdot,t)\geq\beta\}}(x)-\mathds1_{\{d(\cdot,t)<\beta\}}(x),\quad\mbox{for all } x\in B(x_0,r],$

2. $\omega^{\varepsilon,\beta}(x,s)\leq(1-\beta)\mathds1_{\{d(\cdot,s)\geq\beta\}}(x)-\mathds1_{\{d(\cdot,s)<\beta\}}(x),\quad\mbox{for all } x\in \partial B(x_0,r]$, $s\in[t,t+h]$

3. if $(x,s)\in B(x_0,r]\times[t,t+h]$ satisfies $d(x,s)>3\beta$, then
$$\liminfstard\omega^{\varepsilon,\beta}(x,s)\geq1-\beta.$$
\end{lem}
If we assume for the moment that Lemma \ref{lemma propagation} holds true then we can prove Proposition \ref{prop propagation} as a direct consequence (see also \cite{bdl}). In fact, if $d(x,t)\geq\beta>0$, then also $\phi(x,t)>0$ and so, by property (iii) of $\phi$, $x\in\Omega^1_t$. By definition of $(\Omega^1_t)_{t\in(0,T)}$ this means that $\liminfstar u^\varepsilon(x,t)\geq1>1-\beta$ and so there exists an $\varepsilon_{x,t}>0$ such that, for all $\varepsilon\leq\varepsilon_{x,t}$, $(y,\tau)\in B(x,\varepsilon_{x,t})\times(t-\varepsilon_{x,t},t+\varepsilon_{x,t})$, we have $u^\varepsilon(y,\tau)\geq1-\beta$. Thus, by the compactness of $\{x\in B(x_o,r]:\phi(x,t)\geq0\}$ we can select an $\bar\varepsilon>0$, possibly depending only on $\beta$, so that, for all $\varepsilon\leq\bar\varepsilon$, and $x\in\{y\in B(x_0,r]:d(y,t)\geq\beta\}$ we have $u^\varepsilon(x,t)\geq1-\beta$. Therefore
$$u^\varepsilon(x,t)\geq(1-\beta)\mathds1_{\{d(\cdot,t)\geq\beta\}}(x)-\mathds1_{\{d(\cdot,t)<\beta\}}(x).$$
for all $\varepsilon\leq\bar\varepsilon$, $x\in B(x_0,r]$. In a similar way we can also obtain that, for $\varepsilon$ small enough,
$$u^\varepsilon(x,s)\geq(1-\beta)\mathds1_{\{d(\cdot,s)\geq\beta\}}(x)-\mathds1_{\{d(\cdot,s)<\beta\}}(x),$$
for any $(x,s)\in \partial B(x_0,r]\times[t,t+h]$.
Combining these inequalities with those in 1. and 2. in the statement of Lemma \ref{lemma propagation}, by the maximum principle we can conclude that
$$\omega^{\varepsilon,\beta}(x,s)\leq u^\varepsilon(x,s),\quad\mbox{for all }(x,s)\in B(x_0,r]\times[t,t+h].$$
By 3. in Lemma \ref{lemma propagation}, $\liminfstar u^\varepsilon(x,s)\geq1-\beta$ for every $(x,s)\in B(x_0,r]\times[t,t+h]$ such that $d(x,s)>3\beta$. Since $\beta$ is arbitrary we can now send $\beta$ to zero in order to obtain that $\liminfstar u^\varepsilon(x,s)\geq1$ if $(x,s)\in B(x_0,r]\times[t,t+h]$ and $\phi(x,s)>0$. Finally we remark that, if $s\in(t,t+h)$, $x\in B(x_0,r)$ are given and $\phi(x,s)>0$, we have that $\phi(y,\tau)>0$ in a neighborhood of $(x,s)$ and therefore $\liminfstar u^\varepsilon(y,\tau)\geq1$ for $(y,\tau)$ in a neighborhood of $(x,s)$ in $B(x_0,r)\times(t,t+h)$. Thus $x\in\Omega^1_s$.
\end{proof}
\begin{proof}[Proof of Lemma \ref{lemma propagation}] This proof is similar to the one of Lemma \ref{initialization second lemma}, although with a different and not explicit function $\phi$, and therefore we just sketch the main differences. First of all we observe that since $\phi$ satisfies property (ii) of Proposition \ref{prop propagation} there exists $\gamma>0$ such that $d$ is smooth in the set $Q_\gamma=\{(x,s)\in B(x_0,r]\times[t,t+h]:|d(x,s)|\leq\gamma\}$, $|D\phi(x,s)|\neq0$ in $Q_\gamma$. Since $Dd=\frac{D\phi}{|D\phi|}$ and $d_t=\frac{\phi_t}{|D\phi|}$ on $\{\phi=0\}$, and using (i), we may also suppose that
\begin{equation}\label{equation distance}
d_t(x,s)+\bar c(x)\leq-\frac{\tilde C}{4|D\phi(x,s)|}\quad\mbox{for all }(x,s)\in Q_\gamma.
\end{equation}
We notice that for every $\gep$ sufficiently small we have that $c^\gep\leq\overline c$ and will restrict to such values of $\gep$ in the reaction-diffusion equation.

As in Lemma \ref{initialization second lemma} we first define a function $v^{\varepsilon}$ in $Q_\gamma$ as
$v^{\varepsilon}(x,t)=q^{\varepsilon,\delta}\Big(\frac{d(x,t)-2\beta}{\varepsilon},x\Big)-2\beta$, with a suitable auxiliary parameter $\delta\in(0,\bar\delta]$. Thanks to inequality (\ref{equation distance}), the traveling wave equation and (\ref{eqestdelta}), we can see that for $(x,t)\in Q_\gamma$,
\begin{equation*}
\begin{split}
v_t^{\varepsilon}-\varepsilon\Delta v^{\varepsilon}+\frac{f^\varepsilon(v^\varepsilon,x)}{\varepsilon}&\leq\frac{\qued_r}{\varepsilon}\big(-\bar c(x)-\frac{\tilde C}{4|D\phi(x,s)|}+c^{\varepsilon,\delta}(x)-\gep\Delta d\big)-\frac{\delta}{\varepsilon}+\\
&\qquad\qquad\qquad\qquad+2\abs{D\qued_r}+\varepsilon\abs{\Delta\qued}-\frac{2\beta f^\gep_q(\qued,x)}{\gep}+\frac{2\beta^2\norm{f_{qq}^\gep}_\infty}{\gep}\\
&\leq\frac{1}{\gep}\Big[\qued_r\big(M\delta-\frac{\tilde C}{4\norm{D\phi_{|Q_\gamma}}_\infty}+\gep\abs{\Delta d}\big)-2\beta f^\gep_q(\qued,x)+2\beta^2\norm{f_{qq}^\gep}_\infty\Big]\\
&\qquad\qquad\qquad\qquad\qquad\qquad\qquad\qquad\qquad\qquad-\frac{\delta}{\varepsilon}+\frac{2M_1}{\varepsilon^k}+\gep\frac{M_2}{\varepsilon^{2k}} \\
&\leq\frac{1}{\gep}\Big[-\frac{\tilde C}{16\norm{D\phi_{|Q_\gamma}}_\infty}\qued_r-2\beta f^\gep_q(\qued,x)+2\beta^2\norm{f_{qq}^\gep}_\infty\Big]-\frac{\delta}{2\varepsilon},
\end{split}
\end{equation*}
for $\delta>0$ (independent of $\beta$) and then $\varepsilon$ small enough. As in Lemma \ref{initialization second lemma} it can be easily seen that, if we choose $\beta$ small enough and independent of $\delta$, the sum of the terms inside the square brackets is non positive and so $v^\gep$ is a strict subsolution in $Q_\gamma$.
From now on the extension to a global subsolution $\omega^{\varepsilon,\beta}$ in $B(x_o,r]\times[t,t+h]$ and the proof that such a function satisfies 1, 2, 3, is similar to that of Lemma \ref{initialization second lemma} and we omit it.
\end{proof}

The proof of Theorem \ref{thm} is now easy by using the previous two steps and Corollary \ref{Cor levelset-flow}.
\end{proof}

\section{The no-interior condition}
In this section we come back to consider the Cauchy problem (\ref{eqhj}). We want to prove that,  since the velocity $\alpha$ has a constant sign, the zero level set $\{x:u(x,t)=0\}$ of the (unique) continuous viscosity solution of (\ref{eqhj}) has an empty interior provided so does the zero level set of the initial condition $\{x:u_o(x)=0\}$, i.e  condition (\ref{nointerior}) is fullfilled. To this end we use the representation formula for $u$ obtained by the authors in \cite{dzs1}. In order to apply such a result we assume in this section that,
 $u_o\in C(\RRn)$ and $\alpha:\RRn\to[\rho,+\infty)$, for some $\rho>0$, is piecewise Lipschitz continuous as defined in Section 2.
Its discontinuity set $\tilde\Gamma\in\RRn$ is then the finite union of connected Lipschitz hypersurfaces.
Under these hypothesis we have that, if we denote with $(x(s),t(s))=(x(s\,;x,m),t(s\,;x,t,m))$ the Caratheodory solution of the dynamical system
$$\left\{
\begin{array}{l}
\dot x(s)=m(s),\\
t(s)=t-\int_0^s \frac{1}{\alpha_*(x(s))}\;ds,\\
(x(0),t(0))=(x,t),
\end{array}\right.$$
then, for any $(x,t)\in\RRn\times[0,+\infty)$,
\begin{equation}\label{representation formula}
u(x,t)=\inf_{m\in\mathcal C}u_0(x(\tau_{x,t}(m);x,m)).
\end{equation}
Here $\mathcal C$ is the set of all measurable functions $m:[0,+\infty)\to A:=\{a\in\RRn:|a|\leq1\}$ (controls) and $\tau_{x,t}(m)$ satisfies
$t(\tau_{x,t}(m))=0$, i.e.
$$t=\int_0^{\tau_{x,t}(m)}\frac{1}{\alpha_*(x(s\,;x,m))}.$$

In order to prove that $\{(x,t):u(x,t)=0\}$ has empty interior we also suppose on the initial condition that
\begin{equation}\label{initial datum hyp}
\begin{split}
\{u_0>0\}&\neq\emptyset,\quad\{u_0<0\}\neq\emptyset,\\
\Gamma_0=\{u_0=0\}&=\partial\{u_0>0\}=\partial\{u_0<0\}
\end{split}\end{equation}

\begin{thm}Assume that $\alpha$ and $u_0$ satisfy all the assumptions above and (\ref{initial datum hyp}) holds. Then the zero level set $\{(x,t):u(x,t)=0\}$ satisfies the no-interior condition in (\ref{nointerior}).
\end{thm}
\begin{proof}For all $(\hat x,\hat t)\in\RRn\times(0,+\infty)$ define the (bounded) set of reachable points from $(\hat x,\hat t)$ as
$$\reach:=\{x(\tau_{\hat x,\hat t}(m);\hat x,m)\::m\:\in\mathcal C\}.$$
First of all observe that $B(\hat x,\rho\hat t]\subseteq\reach$. If in fact $x\in B(\hat x,\rho\hat t)$ and $x\neq\hat x$, then $x=\hat x+a\abs{x-\hat x}$, with $a=\frac{x-\hat x}{\abs{x-\hat x}}$. We consider the control
$$
\hat m(s)=\left\{\begin{array}{ll}
\frac{x-\hat x}{|x-\hat x|},\quad&\hbox{if }s\leq\abs{x-\hat x},\\
0&\hbox{if }s>\abs{x-\hat x}.
\end{array}\right.
$$
We have that $\tau_{\hat x,\hat t}(\hat m)\geq\rho\hat t\geq\abs{x-\hat x}$ and $x(\tau_{\hat x,\hat t}(\hat m);\hat x,\hat m)=x(\abs{x-\hat x};\hat x,\hat m)=x$, i.e.
$x\in\reach$. Using this inclusion and concatenation of control functions, one can then easily show that for every $h\in (0,\hat t)$
$$\overline\reachtwo\subseteq\overline{\bigcup_{x\in\reachtwo}B(x,\rho\frac{h}{2})}\subseteq \bigcup_{x\in\reachtwo}B(x,\rho h)\subseteq\bigcup_{x\in\reachtwo}\mathcal R_{x,h}
\subseteq\reach,$$
and so
\begin{equation}\label{reach inclusion}
\overline\reachtwo\subseteq\overset{\circ}{\mathcal{R}}_{\hat x,\hat t}\quad\mbox{for all }(\hat x,\hat t)\in\RRn\times[0,+\infty),\;h>0.
\end{equation}

Next we claim that if $u(\hat x,\hat t)=0$ then $u(\hat x,\hat t-h)>0$ for every $h>0$, thus $(\hat x,\hat t)\notin$ Int$\{(x,t):u(x,t)=0\}$. Indeed
suppose that $u(\hat x,\hat t)=0$ and $h>0$. By (\ref{reach inclusion}) and the representation formula (\ref{representation formula}) for $u$ we have that $u(\hat x,\hat t-h)=\inf\{u_0(y):y\in\reachtwo\}\geq u(\hat x,\hat t)=0$. Assume by contradiction that $u(\hat x,\hat t-h)=0$, i.e. there exists $\hat y\in\overline\reachtwo$ such that $u_0(\hat y)=0$. Let $r>0$ be such that $B(\hat y,r)\subseteq\overset{\circ}{\mathcal{R}}_{\hat x,\hat t}$; by (\ref{initial datum hyp}) we have that there exists $y_1\in B(\hat y,r)$ such that $u_0(y_1)<0$. Again, this means that
$$u(\hat x,\hat t)=\inf_{y\in\reach}u_0(y)\leq u_0(y_1)<0,$$
and we get a contradiction since $u(\hat x,\hat t)=0$.

Assuming the claim, our Theorem immediately follows since we have that, for any $(\hat x,\hat t)\in\RRn\times(0,+\infty)$, $h>0$ sufficiently small,
$$\mbox{if }u(\hat x, \hat t)=0,\mbox{ then }\quad u(\hat x, \hat t-h)>0\mbox{ and }u(\hat x, \hat t+h)<0.$$
\end{proof}

\section{A second asymptotic problem}

In this last section we want to briefly discuss a different scaling in the reaction-diffusion equation, namely (\ref{asymptotic-probl-two}).
To this end we have to modify some of the assumptions of Section 3.
We have a cubic function $f^\gep$ with the same structure as in section \ref{f hyp} but with (\ref{fproperty2}) replaced by the stronger condition, this time for some $k\in[0,1)$,
\begin{equation}\label{fproperty2b}
\begin{split}
\left\{
\begin{array}{l}
\mbox{for every compact } K \subset\RR\mbox{ there exists a constant }C=C(K)>0\\
\mbox{ such that, for all }(q,x)\in K\times\RRn,\:1\leq i,j\leq n,\\
|f^\varepsilon_{q}(q,x)|,|f^\varepsilon_{qq}(q,x)|\leq C,\;|f^\varepsilon_{x_i}(q,x)|,|f^\varepsilon_{x_iq}(q,x)|\leq\frac{C_1}{\varepsilon^{k-1}},\;|f^\varepsilon_{x_ix_j}(q,x)|\leq\frac{C_2}{\varepsilon^{2k-1}}.
\end{array}
\right.
\end{split}
\end{equation}
Moreover we assume that now $\overline m=\underline m=0$, so that $m_o^\gep\longrightarrow0^+$ locally uniformly in $\RRn$, i.e. for any given compact $K_1\subset\R^n$ and all $\sigma>0$ we can find an $\varepsilon_\sigma>0$ such that $m_o^{\gep}(x)\in(0,\sigma]$ for all $\gep\leq\gep_{\sigma}$, $x\in K_1$. Finally in (\ref{eqestf}) we choose $m_1=\sigma$, $m_2=0$.

Consequently we adapt the growth rate in (\ref{travelling wave properties}) as
\begin{equation}\label{travelling wave propertiesb}
\begin{array}{l}
|Dq^\varepsilon(r,x)|,\;|Dq^\varepsilon_r(r,x)|\leq \frac{M_1}{\varepsilon^{k-1}},\;|D^2 q^\varepsilon(r,x)|\leq \frac{M_2}{\varepsilon^{2k-1}},\;\mbox{for all }x\in K_1,\,r\in\RR.
\end{array}\end{equation}
During the proofs we also need to modify the cubic-like function $f^\gep$ as $f^{\gep,\delta}=f^\gep+\gep\delta$, for $\delta\in[-\overline \delta,\overline\delta]$ and modify accordingly the notations for the properties of $f^{\gep,\delta}$. Moreover we assume that there is a constant $M>0$ independent of $\gep,\delta$ such that
\begin{equation}\label{eqestdeltab}
\sup_{x\in\R^n}\left[|c^\gep(x)-c^{\gep,\delta}(x)|+|1-m_+^{\gep,\delta}(x)|+|1+m_-^{\gep,\delta}(x)|\right]\leq M|\delta|\gep.
\end{equation}
As for the asymptotics of the velocity of the traveling wave solutions, we replace (\ref{c epsilon estimate}) by
\begin{equation}\label{c epsilon estimate 2}\begin{array}{c}
0<2\rho\leq{n_1(x)}<\frac{c^\varepsilon(x)}{\gep}< {n_2(x)}\leq2(1-\rho),\quad\mbox{for any }x\in\RRn,\\
\frac{c^{\varepsilon}}{\gep}\longrightarrow\alpha,\quad \hbox{ locally uniformly off }\tilde\Gamma,
\end{array}\end{equation}
where the functions $\alpha,n_1,n_2$ are assumed as in (\ref{alpha def}) and $\tilde\Gamma$ is a smooth hypersurface.

We want to show that the front associated with the asymptotics of (\ref{asymptotic-probl-two}) evolves according to the geometric pde (\ref{eqmc}),
whose normal velocity is given by ${\cal K}-\alpha$, where $\cal K$ is the mean curvature of the front.
To be more precise we introduce the following sets
\begin{equation}\label{new asymp}
\begin{array}{ll}
\Omega^1:=\Int\Big\{(x,t)\in\RRn\times[0,T]:\displaystyle\liminfstard \frac{u^\epsilon(x,t)-1}{\gep}= 0\Big\},\\
\Omega^2:=\Int\Big\{(x,t)\in\RRn\times[0,T]:\displaystyle\limsupstard \frac{u^\epsilon(x,t)+ 1}{\gep}=0\Big\}
\end{array}
\end{equation}
and $\Omega_o^1,\;\Omega_o^2$ as before. The result is as follows.

\begin{thm}\label{thm2} Assume (\ref{alpha def}), (\ref{fproperty}), (\ref{fproperty2b}), (\ref{eqestf}), (\ref{travelling wave properties a}), (\ref{travelling wave properties b}), (\ref{travelling wave propertiesb}), (\ref{eqestdeltab}), (\ref{c epsilon estimate 2}). Let $u^\varepsilon$ be the unique solution of (\ref{asymptotic-probl-two}), where $g:\RRn\to[-1,1]$ is a continuous function such that the sets $\Gamma_o=\{x:g(x)=0\}$, $\Omega_o^+=\{x:g(x)>0\}$, $\Omega_o^-=\{x:g(x)<0\}$ are nonempty and $(\Gamma_o,\Omega_o^+,\Omega_o^-)\in{\mathcal E}$. Then
$$u^\varepsilon(x,t)\longrightarrow
\left\{\begin{array}{rll}
1& \mbox{in}&\{(x,t):u(x,t)>0\},\\
-1& \mbox{in}&\{(x,t):u(x,t)<0\},
\end{array}
\right.
$$
locally uniformly as $\varepsilon\to0$,
provided $u$ is a unique, continuous viscosity solution of
\begin{equation}\label{PCfront2}
\left\{\begin{array}{ll}
u_t(x,t)+F(Du(x,t),D^2u(x,t))+\alpha(x)|Du(x,t)|=0\mbox{ in }\RRn\times(0,+\infty),\\
u(x,0)=d_o(x),
\end{array}\right.
\end{equation}
and the comparison principle holds for (\ref{PCfront2}). Here
$d_o$ is the signed distance to $\Gamma_o$ which is positive in $\Omega^+_o$ and negative in $\Omega_o^-$. If in addition the no-interior condition (\ref{nointerior}) for the set $\{u=0\}$ holds, then, as $\varepsilon\to0$,
$$u^\varepsilon(x,t)\longrightarrow
\left\{\begin{array}{rll}
1&\mbox{in} &\{u>0\},\\
-1&\mbox{in} &\overline{\{u>0\}}^c,
\end{array}
\right.
$$
 locally uniformly.
\end{thm}
\begin{rem} Comparison principle and uniqueness for equation (\ref{PCfront2}) is not fully known at the moment, as far as we know. We proved in \cite{dzs3} that a comparison principle holds when (\ref{PCfront2}) is considered in a bounded domain with a prescribed Dirichlet boundary condition.
\end{rem}

\begin{proof} The proof follows the same steps as the one of Theorem \ref{thm}, so we just point out the main changes. Consider two families of open sets of $\RRn$ $(\Omega^1_t)_{t\in[0,T)}$ and $(\Omega^2_t)_{t\in[0,T)}$  defined as in (\ref{omega1-2t}), (\ref{omega0}) with $\Omega^1$ and $\Omega^2$ as in (\ref{new asymp}). By the maximum principle $-1\leq u^\gep\leq1$.

\emph{First step: initialization.} We want to show that $\Omega_0^+=\{d_o>0\}\subseteq\Omega_0^1$ and $\Omega_0^-=\{d_o<0\}\subseteq\Omega_0^2$. For the first inclusion we consider $\hat x\in\{x:d_o(x)>0\}$ and find $r,\sigma>0$ such that
\begin{equation*}
\begin{array}{lll}
g(x)&\geq5\sigma&\mbox{ for all }x\in B(\hat x, r)\\
&\geq c^\gep(x)+4\sigma&\mbox{for all }x\in B(\hat x,r),\,\gep\leq\gep_\sigma.
\end{array}
\end{equation*}
and
$$m_o^\gep(x)\in(0,\sigma],\quad\mbox{for all }x\in\RRn,\,\gep\leq\gep_\sigma.$$
 This means in particular that
\begin{equation}\label{lower estimate uepsilonb}
u^\varepsilon(x,0)=g(x)\geq5\sigma\mathds1_{B(\hat x,r)}(x)-\mathds1_{B(\hat x,r)^c}(x).
\end{equation}
We define the function $\Phi:\RRn\times[0,T]\to\RR$, and the signed distance function $d$ to to $\{x:\Phi(x,t)=0\}$  as in (\ref{Phi def}).

Now we state the analogous of Lemma \ref{initialization first lemma} and of Lemma \ref{initialization second lemma}
\begin{lem}\label{initialization first lemma2}Under the same assumptions of Theorem \ref{thm2} we have that for any $\beta>0$ there exist $\tau=\tau(\beta)>0$ and $\bar\varepsilon=\bar\varepsilon(\beta)$ such that, for all $0<\varepsilon\leq\bar\varepsilon$, we have
$$u^\varepsilon(x,t_\varepsilon)\geq(1-\beta\gep)\mathds1_{\{d(\cdot,0)\geq\beta\}}(x)-\mathds1_{\{d(\cdot,0)<\beta\}}(x),\quad x\in\RRn,$$
where $t_\varepsilon=\tau\varepsilon^2|\lg\gep|$ and $d(x,t)=\sqrt{(r^2-Ct)^+}-|x-\hat x|$.
\end{lem}

\begin{lem}\label{initialization second lemma2}There exist $\bar h=\bar h(r,\hat x)$, $\bar\beta=\bar\beta(r,\hat x)>0$ independent of $\gep$ such that if $\beta\leq\bar\beta$ and $\varepsilon\leq\bar\varepsilon(\beta)$, then there exists a subsolution $\omega^{\varepsilon,\beta}$ of (\ref{asymptotic-probl-two}-i) in $\RRn\times(0,\bar h)$ that satisfies
$$\omega^{\varepsilon,\beta}(x,0)\leq(1-\beta\gep)\mathds1_{\{d(\cdot,0)\geq\beta\}}(x)-\mathds1_{\{d(\cdot,0)<\beta\}}(x),\quad x\in\RRn.$$
If  moreover $(x,t)\in B(\hat x,r)\times(0,\bar h)$ and $d(x,t)>3\beta$, then $$\liminfstard\frac{\omega^{\varepsilon,\beta}(x,t)-1}{\gep}\geq-2\beta.$$
\end{lem}
\begin{proof}[Proof of Lemma $\ref{initialization first lemma2}$] Let $\beta>0$ fixed. From now on we restrict $\gep$ to $\gep\leq\gep_\sigma$. To prove our thesis we have to modify the function $f^\gep$ as in \cite{xce,bdl}. Let $\bar f\in C^2(\R\times\R^n)$ be a function as in (\ref{eqestf}) with $m_1=2\sigma$. Consider a smooth cut-off $\rho\in C^{\infty}_0(\R)$ such that $0\leq\rho\leq1$, $\rho(s)=1$ if $|s|\leq1$ and $\rho(s)=0$ if $|s|\geq2$. Assume moreover that $\rho$ satisfies $-2\leq s\rho'(s)\leq0$ and $|\rho''(s)|\leq4$ for all $s\in\R$. Now define two further smooth functions $\rho_1,\rho_2:\R\to[0,1]$ as
$$
\rho_1(q)=\rho\Big(\frac{q-2\sigma}{\sigma}\Big)\qquad\rho_2(q)=\rho\Big(\frac{4(q-2\sigma)}{\sigma}\Big)
$$
and set
$$\bar f^\gep(q,x)=(1-\rho_1(q))f^\gep(q,x)+\rho_1(q)\bar f(q)$$
and
$$\tilde f^\gep(q,x)=(1-\rho_2(q))\bar f^\gep(q,x)+\rho_2(q)\frac{2\sigma-q}{|\lg\gep|}.$$
Notice that for any $x\in\RRn$, $\tilde f^\gep(\cdot,x)$ has $\{-1,2\sigma,1\}$ as zeros and satisfies properties similar to $f^\gep$. Moreover $\tilde f^\gep$ does not depend on $x$ for all $q\in[\sigma,3\sigma]$ and $f^\gep\leq\min\{\bar f^\gep,\tilde f^\gep\}$.

1. As in Chen \cite{xce}, if we denote by $\chi=\chi(\tau,\xi;x)\in C^2([0,+\infty)\times\RR\times\RRn)$ the solution of
\begin{equation}\label{ordinary equation chi tilde}
\left\{
\begin{array}{l}
\dot\chi(\tau,\xi;x)+\tilde f^\varepsilon(\chi(\tau,\xi;x),x)=0,\quad \tau>0,\\
\chi(0,\xi;x)=\xi,
\end{array}
\right.
\end{equation}
it follows that $\chi$ satisfies property (\ref{property one chi}) in the proof of Lemma \ref{initialization first lemma} while properties (\ref{property two chi}) and (\ref{property three chi}) are replaced by the following: for all $\beta,\;\sigma>0$
there exist $\tau_o=\tau_o(\beta,\sigma),\gep_o=\gep_o(\beta,\sigma)>0$ such that, for all $\tau\geq\tau_o|\log\gep|$ and $\gep\leq\gep_o$
\begin{equation}\tag{$\tilde\chi2$}\label{property two tilde chi}
\begin{array}{l}
\chi(\tau,\xi;x)\geq1-\beta\gep\,\quad\forall\,\xi\geq 4\sigma.
\end{array}
\end{equation}
Moreover, since for any $C>1$ we have that $\chi(\tau,\xi,x)\in[-C,C]$ for all $\xi\in[-C,C]$, $\tau\geq0$, $x\in\RRn$, it also holds that
for any $C>1$, $a>0$ there exists a constant $M_{C,a}>0$ such that
\begin{equation}\tag{$\tilde\chi3$}\label{property three tilde chi}
\begin{array}{l}
|\chi_{\xi\xi}(\tau,\xi;x)|\leq \frac{M_{C,a}}{\gep}\chi_\xi(\tau,\xi;x),\quad
|\chi_{x_i}(\tau,\xi;x)|,\leq \frac{M_{C,a}}{\gep^{k-1}}\\
|\chi_{\xi x_i}(\tau,\xi;x)|\leq \frac{M_{C,a}}{\gep^{k-1}}\chi_\xi(\tau,\xi;x),\quad
|\chi_{x_ix_i}(\tau,\xi;x)|\leq \frac{M_{C,a}}{\gep^{2k-1}}\chi_\xi(\tau,\xi;x),
\end{array}
\end{equation}
for any $\tau\leq a|\ln\gep|$, $\xi\in[-C,C],\,x\in\RRn,\,i\in\{1,2,\cdots,n\}$ and $\gep$ small enough.

2. Consider a smooth nondecreasing function $\psi$  such that $\psi(z)=-1$ if $z\leq0$ and $\psi(z)=5\sigma$ if $z\geq\beta\wedge\frac{\sigma}{2}$. Similarly as before, the function
$$\underline u^\gep(x,t)=\chi\Big(\frac{t}{\gep^2},\psi(d(x,0))-\frac{Kt}{\gep},x\Big)$$
satisfies $\underline u^\gep(x,0)\leq u^\gep(x,0)$. Moreover it is a subsolution of (\ref{asymptotic-probl-two}-i) in $\RRn\times(0,\tau_o\gep^2|\lg\gep|)$.
Indeed we can compute by (\ref{property three tilde chi}),
\begin{equation*}\begin{array}{ll}
\underline u_t^\varepsilon-\Delta\underline u^\varepsilon+\frac{f^\varepsilon(\underline u^\varepsilon,x)}{\varepsilon^2}&= \frac{\dot\chi+f^\gep(\chi,x)}{\gep^2}-K\frac{\chi_\xi}{\gep}-\chi_{\xi\xi}(\psi')^2-\chi_\xi(\psi''+\psi'\Delta d)\\
&\quad+2\psi'D\chi_\xi\cdot Dd+\Delta\chi\\
&=\frac{f^\gep(\chi,x)-\tilde f^\gep(\chi,x)}{\gep^2}+\frac{\chi_\xi}{\gep}[-K-\gep(\psi''+\psi'\Delta d)+\\
&\quad +M_{2,\tau_0}((\psi')^2+\gep^{2-k}\psi')+\gep^{2-2k}]\\
&\leq-\frac{\chi_\xi}{\gep}\big(K-M_{2,\tau_0}\norm{\psi'}_\infty^2+o_\gep(1)\big)\leq0,
\end{array}\end{equation*}
for $K$ large enough. Therefore using the maximum principle and property (\ref{property two tilde chi}) we can prove that $u^\gep(x,t_\gep)\geq1-\beta\gep$ if $t_\gep=\tau_o\gep^2|\lg\gep|$ and $d(x,0)\geq\beta$ (from which Lemma \ref{initialization first lemma2} follows).
\end{proof}

\begin{proof}[Proof of Lemma $\ref{initialization second lemma2}$] The construction of a subsolution that satisfies this Lemma is very similar to the one in Lemma \ref{initialization second lemma}. Let $\Phi$, $d$ and $Q_{\gamma,\bar h}$ defined as in (\ref{Phi def}) where now the fixed constant $C$ satisfies
\begin{equation*}
C\geq 2r\Big[\frac{n-1}\gamma+4\Big].
\end{equation*}
The construction of our subsolution $\omega^{\gep,\beta}$ follows the usual steps. We first define for any $(x,t)\in Q_{\gamma,\bar h}$  $$v^{\varepsilon}(x,t)=q^{\varepsilon,\delta}\Big(\frac{d(x,t)-2\beta}{\varepsilon},x\Big)-2\beta\gep,$$
where $q^{\gep,\delta}$ is the solution of the travelling wave equation (\ref{travelling-wave eq}) with $f^\gep$ replaced by $f^{\gep,\delta}=f^\gep+\gep\delta$. The function $v^\gep$ is a subsolution of (\ref{asymptotic-probl-two}-i) in $Q_{\gamma,\bar h}$. Indeed,
\begin{equation*}
\begin{array}{ll}
v_t^{\varepsilon}-\Delta v^{\varepsilon}+\frac{f^\varepsilon(v^\varepsilon,x)}{\varepsilon^2}&=
\frac{\qued_r d_t}{\varepsilon}-\frac{\qued_{rr}}{\varepsilon}-\frac{2}{\gep}D\qued_r\cdot Dd-\frac{\qued_r}{\gep}\Delta d-\Delta\qued+\frac{f^\varepsilon(\qued-2\beta,x)}{\varepsilon^2}\\
&\quad-\frac{2\beta}{\gep} f^\gep_q(\qued,x)+2\beta^2\gep\norm{f_{qq}^\gep}_\infty\\
&\leq\frac{1}{\gep}\Big[-\qued_r-2\beta f^\gep_q(\qued,x)+2\beta^2\gep\norm{f_{qq}^\gep}_\infty\Big]+\left[-\frac{\delta}{\gep}+2\frac{M_1}{\gep^k}+\frac{M_2}{\gep^{2k-1}}\right],
\end{array}
\end{equation*}
and then we conclude as before. The extension of $v^\gep$ to a subsolution in the entire strip $\RRn\times[0,\bar h]$ proceed now similarly to the one in Lemma \ref{initialization second lemma}.  We first prove that the function $\bar v^\gep:\{(x,t)\in\RRn\times[0,\bar h]:d(x,t)\leq\gamma\}\to\RR$, defined as
$$\bar v^\varepsilon(x,t)=
\left\{\begin{array}{lll}
\sup(v^\varepsilon(x,t),-1)&\mbox{if}&-\gamma<d(x,t)\leq\gamma,\\
-1&\mbox{if}&d(x,t)\leq-\gamma,
\end{array}
\right.$$
is a subsolution of (\ref{asymptotic-probl-two}-i). Eventually we define our subsolution $\omega^{\gep,\beta}$ as
$$
\omega^{\varepsilon,\beta}(x,t)=
\left\{\begin{array}{ll}
\psi(d(x,t))\bar v^{\varepsilon}(x,t)+(1-\psi(d(x,t)))(1-\gep\beta)&\mbox{if }d(x,t)<\gamma,\\
1-\gep\beta&\mbox{if }d(x,t)\geq\gamma.
\end{array}
\right.
$$
for $(x,t)\in\RRn\times[0,\bar h]$. These proofs do not contain any new ideas with respect to the ones in Lemma \ref{initialization second lemma} and we omit them. 
\end{proof}

\emph{Second step: propagation.} The proof of the fact that $\omegaone$ and $((\Omega^2_t)^c)_{t\in(0,T)}$ are respectively super and subflows with normal velocity ${\cal K}-\alpha$, where $\cal K$ is the mean curvature of the level set, is very close to the one in Theorem \ref{thm}. Here we approximate our discontinuous limit velocity $\alpha$ with the following continuous functions:
$$\hat c{^\gep}(x):=\eta^\gep(x)n_2(x)+(1-\eta^\gep(x))\frac{c^\gep(x)}{\gep},\quad
\check c{^\gep}(x):=\xi^\gep(x)n_1(x)+(1-\xi^\gep(x))\frac{c^\gep(x)}{\gep},$$
with $\eta^\gep$ and $\xi^\gep$ as in Theorem \ref{thm}. If we put ${\hat{\mathcal F}}=\{\hat c^\gep,\,\varepsilon>0\}$, ${\check{\mathcal F}}=\{\check c^\gep,\,\varepsilon>0\}$, then Proposition \ref{equivalent def of generalized flow} takes the following form.
\begin{prop}\label{equivalent def of generalized flow2}
\begin{description}
\item[(i)] A family $\open$ of open subsets of $\RRn$ such that the set $\Omega:=\bigcup_{t\in(0,T)}\Omega_t\times\{t\}$ is open in $\RRn\times[0,T]$ is a \emph{generalized superflow} with normal velocity $-F-\alpha$ if and only if
it is a generalized superflow with normal velocity $-F- \hat c\in C(\R^n)$, for all $\hat c\in\hat{{\mathcal F}}$;
\item[(ii)] A family $\close$ of close subsets of $\RRn$ such that the set $\mathcal{F}:=\bigcup_{t\in(0,T)}\mathcal{F}_t\times\{t\}$ is closed in $\RRn\times[0,T]$ is a \emph{generalized subflow} with normal velocity $-F-\alpha$ if and only if
it is a generalized subflow with normal velocity $-F-\check c$, for all $\check c\in{\check{\mathcal F}}$.
\end{description}
\end{prop}
The modifications that we need in this proof follow the lines of the previous Lemma.
\end{proof}

\end{document}